\documentclass[reqno, 12pt]{amsart}

\pdfoutput=1

\usepackage{enumerate}
\usepackage{latexsym}
\usepackage[centertags]{amsmath}
\usepackage{amsfonts}
\usepackage{amssymb}
\usepackage{amsthm}
\usepackage{newlfont}
\usepackage{graphics}
\usepackage{color}
\usepackage{float}
\usepackage{mathrsfs}
\usepackage{diagbox}
\usepackage{hyperref}
\usepackage{longtable}
\usepackage{rotating}
\usepackage{multirow}
\usepackage{extarrows}
\usepackage{cite}
\usepackage[sort,compress,numbers]{natbib}

\usepackage[utf8]{inputenc}

\newtheorem{thm}{Theorem}[section]
\newtheorem{cor}[thm]{Corollary}
\newtheorem{lem}[thm]{Lemma}
\newtheorem{prop}[thm]{Proposition}

\theoremstyle{mydefinition}
\newtheorem{dfn}[thm]{Definition}

\theoremstyle{myremark}
\newtheorem{rem}[thm]{Remark}
\newtheorem{exa}[thm]{Example}

\allowdisplaybreaks[4]

\title[Ehrhart Polynomials of Order Polytopes]{Ehrhart Polynomials of Order Polytopes: Interpreting Combinatorial Sequences on the OEIS}

\author{Feihu Liu$^{1}$, Guoce Xin$^{2, *}$, and Chen Zhang$^{3}$}

\address{$^{1, 2, 3}$School of Mathematical Sciences, Capital Normal University, Beijing 100048, PR China}
\email{$^1$\texttt{liufeihu7476@163.com}\ \  \& $^2$\texttt{guoce\_xin@163.com}\ \  \& $^3$\texttt{ch\_enz@163.com}}
\date{December 25, 2024}
\thanks{$*$ This work was partially supported by NSFC(12071311).}
\begin{document}
\maketitle

\begin{abstract}
In this paper, we provide an overview of Ehrhart polynomials associated with order polytopes of finite posets, a concept first introduced by Stanley. We focus on their combinatorial interpretations for many sequences listed on the OEIS. We begin by exploring the Ehrhart series of order polytopes resulting from various poset operations, specifically the ordinal sum and direct sum. We then concentrate on the poset $P_\lambda$ associated with the Ferrers diagram of a partition $\lambda = (\lambda_1, \lambda_2, \ldots, \lambda_t)$. When $\lambda = (k, k-1, \ldots, 1)$, the Ehrhart polynomial is a shifted Hankel determinant of the well-known Catalan numbers; when $\lambda = (k, k, \ldots, k)$, the Ehrhart polynomial is solved by Stanley's hook content formula and is used to prove conjectures for the sequence [A140934] on the OEIS. When solving these problems, we rediscover Kreweras' determinant formula for the Ehrhart polynomial $\mathrm{ehr}(\mathcal{O}(P_{\lambda}), n)$ through the application of the Lindstr\"om-Gessel-Viennot lemma on non-intersecting lattice paths.
\end{abstract}

\noindent
\begin{small}
 \emph{Mathematics subject classification}: Primary 05A15; Secondary 05A17, 06A07, 52B11, 52B20.
\end{small}

\noindent
\begin{small}
\emph{Keywords}: Integer sequence; Order polytope; Ehrhart polynomial; Plane partition; Non-intersecting lattice path; Jacobi-Trudi identity.
\end{small}

\section{Introduction}
In the paper, we always use $\mathbb{R}$, $\mathbb{Z}$, $\mathbb{N}$, and $\mathbb{P}$ to represent the set of all real numbers, integers, non-negative integers, and positive integers, respectively.

Let $\mathcal{P}$ be a $d$-dimensional convex polytope in $\mathbb{R}^{N}$. Assume $\mathcal{P}$ is \emph{integral}, i.e., it has only integer vertices. The function
\[
\mathrm{ehr}(\mathcal{P},n)=|n\mathcal{P}\cap \mathbb{Z}^d|,\ \ \ \ \ n=1,2,\ldots,
\]
counts the number of integer points in the $n$-th dilation of $\mathcal{P}$, where $n\mathcal{P}:=\{n \alpha : \alpha \in \mathcal{P}\}$. Ehrhart \cite{Ehrhart62} proved that $\mathrm{ehr}(\mathcal{P},n)$ is a polynomial in $n$ of degree $d$, called the \emph{Ehrhart polynomial} of $\mathcal{P}$. Furthermore, the coefficient of $n^d$ in $\mathrm{ehr}(\mathcal{P},n)$ is equal to the (relative, becomes Euclidean when $N=d$) volume of $\mathcal{P}$, the coefficient of $n^{d-1}$ equals half of the boundary volume of $\mathcal{P}$, and the constant term is always $1$ (see \cite{BeckRobins}).

Define the \emph{Ehrhart series} of $\mathcal{P}$ to be the generating function
\[
\mathrm{Ehr}(\mathcal{P},x)=1+\sum_{n\geq 1}\mathrm{ehr}(\mathcal{P},n)x^n.
\]
It is of the form
\[
\mathrm{Ehr}(\mathcal{P},x)=\frac{h^{*}(x)}{(1-x)^{d+1}},
\]
where $d = \dim\mathcal{P}$ and $h^{*}(x)$ is a polynomial in $x$ of degree $\deg h^{*}(x)\leq d$.

We assume basic knowledge on posets (see \cite{RP.Stanley}). Let $(P,\preceq)$ be a finite \emph{partially ordered set} (or \emph{poset}) on $[p]:=\{1,2,\ldots,p\}$. We say that $a$ \emph{covers }$b$ if $b\prec a$ and there is no $c\in P$ such that $b\prec c\prec a$. The \emph{Hasse diagram} of a finite poset $P$ is the graph whose vertices are the elements of $P$, whose edges are the cover relations, and such that if $b\prec a$ then $a$ is drawn above $b$.

Order polytopes associated with posets were first introduced and studied by Stanley. 
\begin{dfn}[Stanley \cite{StanleyDCG}]\label{Definitorderp}
Given a poset $P$ on the set $[p]$, we associate a polytope $\mathcal{O}(P)$, called the \emph{order polytope} of $P$. It is the polytope in $\mathbb{R}^p$ defined by
\begin{align*}
x_i &\leq x_j \qquad \text{if} \qquad i\prec_{P} j;
\\ 0\leq x_i&\leq 1 \qquad \text{for} \qquad  1\leq i\leq p.
\end{align*}
\end{dfn}

The order polytope is an integral polytope, and $\dim \mathcal{O}(P)=p$ \cite{StanleyDCG}. The order polytopes of various special posets have been studied by many scholars. For example, the zig-zag poset \cite{Coons23} and generalized snake poset (including the snake poset and ladder poset) \cite{BellBraun22}. Many authors have generalized order polytopes. For instance, double poset polytopes \cite{Chappell17,HibiTsuchiya17}, marked poset polytopes \cite{ArdilaBliem11}, and enriched poset polytopes \cite{Okada24}.

Let $P$ be a finite poset of cardinality $p \in \mathbb{P}$. The set $[p]$ with its usual order forms a poset with $p$ elements. It is denoted $\mathbf{p}$. We identify a linear extension $\sigma : P\rightarrow \mathbf{p}$ of $P$ with a permutation $\pi=\sigma^{-1}(1)\sigma^{-1}(2)\cdots\sigma^{-1}(p)$ of the set of labels of $P$. The set of all permutations of $[p]$ obtained in this way is denoted $\mathcal{L}(P)$ and is called the \emph{Jordan-H\"older set} of $P$.

The \emph{order polynomial} $\Omega_{P}(m)$ \cite{RP.Stanley} is defined as the number of order-preserving maps $\tau: P\rightarrow \mathbf{m}$. This is a polynomial function in $m$ of degree $p$. The generating function of $\Omega_{P}(m)$ \cite{RP.Stanley} is given by
\[
\sum_{m\geq 1}\Omega_{P}(m)x^m=\frac{\sum_{\pi \in \mathcal{L}(P)}x^{1+d(\pi)}}{(1-x)^{p+1}},
\]
where $d(\pi)$ is the number of descents of permutation $\pi$.

Stanley \cite{StanleyDCG} showed that the Ehrhart polynomial of $\mathcal{O}(P)$ is given by
\[
\mathrm{ehr}(\mathcal{O}(P),m)=\Omega_{P}(m+1).
\]
Therefore, we have
\begin{equation}\label{EhrJordanH}
\mathrm{Ehr}(\mathcal{O}(P),x)=\frac{\sum_{\pi \in \mathcal{L}(P)}x^{d(\pi)}}{(1-x)^{p+1}}.
\end{equation}

In this paper, we focus on the enumerative properties of families of order polytopes. Firstly, we consider the effect on the Ehrhart series (or polynomials) under several operations on posets. For the ordinal sum of the posets, the corresponding Ehrhart series is almost related by product. It is closely related to the Riordan array (see Section 2). For the direct sum of the posets, the corresponding Ehrhart series are related by the Hadamard product. For the dual of poset $P$, the Ehrhart series remains unchanged. For the ordinal product and direct product of two posets, the calculation of Ehrhart is difficult.

Secondly, we consider a class of posets and their variants defined in Section 4. The numerator of the Ehrhart series is given by the descent of multipermutations on multiset (see Theorem \ref{rkrkline}). In some special cases, they have fascinating combinatorial properties (see Propositions \ref{PropEhrCase2}, \ref{sl2set00}, \ref{sl2set1}, \ref{sl2set2}, and \ref{sl2set3}).

Thirdly, we consider the order polytope $\mathcal{O}(P_{\lambda})$ of the poset $P_\lambda$ related to the Ferrers diagram (see Section \ref{sec-Ferrers}). For $\lambda=(\lambda_1,\lambda_2,\ldots,\lambda_t)$, the Ehrhart polynomial $\mathrm{ehr}(\mathcal{O}(P_\lambda),m)$ of the order polytope $\mathcal{O}(P_{\lambda})$ counts the number of plane partitions such that the $i$-th row has at most $\lambda_i$ parts, and the largest part is less than or equal to $n$. We reobtain a determinant formula for the Ehrhart polynomial $\mathrm{ehr}(\mathcal{O}(P_{\lambda}),n)$ for general $\lambda$ using the Lindstr\"om-Gessel-Viennot lemma \cite{GesselViennot,Lindstrom} on non-intersecting lattice paths. The determinant formula for this counting problem is first given by Kreweras \cite[Section 2.3.7]{G.Kreweras}. As applications, we study two special cases as follows.

\begin{enumerate}
  \item When $\lambda=(k, k-1,\ldots ,1)$, the Ehrhart polynomial $\mathrm{ehr}(\mathcal{O}(P_{\lambda}),n)$ is a shifted Catalan Hankel determinant (see Theorem \ref{CnjOnekk-1}).

  \item The rectangular case $\lambda=(k,k,\ldots k)=(k^t)$. An equivalent definition is referred to in \cite{SantosStump17}. When $\lambda=(k,k)$, it is called the ladder poset in \cite{BellBraun22}, where the authors studied the volume of $\mathcal{O}(P_{(k,k)})$. We show that the numerator of $\mathrm{Ehr}(\mathcal{O}(P_{(k,k)}),x)$ is related to the Narayana numbers. We also prove some conjectures in \cite[A140934]{Sloane23} related to $\mathcal{O}(P_{\lambda})$ (see Theorem \ref{hhsnthmschu}, Proposition \ref{Proposit1409}, and Theorem \ref{EhrhOkk-Naray}).
\end{enumerate}

Finally, we define a class of posets, denoted $\bar{P}_{(k,k)}$. The Ehrhart polynomial $\mathrm{ehr}(\mathcal{O}(\bar{P}_{(k,k)}),n)$ is related to the Stirling number of the second kind (see Theorem \ref{stirlssnk}). The numerator of $\mathrm{Ehr}(\mathcal{O}(\bar{P}_{(k,k)}),x)$ is related to the second-order Eulerian triangle (see Theorem \ref{stirlssnkkEuler}).

Another contribution of this paper is to give combinatorial interpretations of many sequences on the OEIS website \cite{Sloane23} using order polytopes.

This paper is organized as follows. In Sections \ref{sec-sum} and \ref{sec-prod}, we study the Ehrhart series when some operations are performed on the posets, that is, direct sum, ordinal sum, direct product, ordinal product, and dual. In Section \ref{sec-Ipr}, we consider a class of posets and their variants. The numerator of its Ehrhart series is given by the descent of multipermutations on a multiset. In Section \ref{sec-Ferrers}, we provide a determinant formula for $\mathcal{O}(P_{\lambda})$, which arises from two different combinatorial interpretations. We also prove some conjectures in \cite[A140934]{Sloane23}. In Section \ref{sec-Pkk}, we study the Ehrhart polynomial and Ehrhart series of $\mathcal{O}(\bar{P}_{(k,k)})$.

\section{Dual, Direct Sum, and Ordinal Sum}\label{sec-sum}
In this section, we introduce three operations that can be performed on posets. The Ehrhart polynomial of the resulting new poset, after applying each of these operations, can be determined by the Ehrhart polynomials of the original posets.

\subsection{Preliminaries}

We first consider the \emph{dual} of a poset $P$. This is the poset $P^*$ on the same set as $P$, but $s\preceq t$ in $P^*$ if and only if $t\preceq s$ in $P$. If $P$ and $P^*$ are isomorphic, then $P$ is called \emph{self-dual}. The following result is obvious.
\begin{prop}
Let $P$ be a finite poset of cardinality $p$. Then
\[
\mathrm{Ehr}(\mathcal{O}(P),x)=\mathrm{Ehr}(\mathcal{O}(P^*),x).
\]
\end{prop}

Secondly, the \emph{direct sum} of $P_1$ and $P_2$ is the poset $P_1+ P_2$ on the disjoint union $P_1 \uplus P_2$ such that $s\preceq t$ in $P_1+ P_2$ if (i) $s,t\in P_1$ and $s\preceq t$ in $P_1$, or (ii) $s,t\in P_2$ and $s\preceq t$ in $P_2$. Obviously, the Ehrhart polynomial of $\mathcal{O}(P_1+P_2)$ is given by
\begin{equation}\label{DirecSumHP}
\mathrm{ehr}(\mathcal{O}(P_1+P_2),n)=\mathrm{ehr}(\mathcal{O}(P_1),n)\cdot \mathrm{ehr}(\mathcal{O}(P_2),n).
\end{equation}

\begin{prop}\label{PropDirecSum}
Let $P_1$ and $P_2$ be two posets. The Ehrhart series of $\mathcal{O}(P_1+P_2)$ is given by
\[
\mathrm{Ehr}(\mathcal{O}(P_1+P_2),x)=\mathrm{Ehr}(\mathcal{O}(P_1),x)\ast \mathrm{Ehr}(\mathcal{O}(P_2),x)=\sum_{n\geq 0}\mathrm{ehr}(\mathcal{O}(P_1),n)\cdot \mathrm{ehr}(\mathcal{O}(P_2),n)x^n,
\]
where we use $\ast$ to denote the \emph{Hadamard product} (see \cite{RP.Stanley}).
\end{prop}

Finally, we introduce the \emph{ordinal sum} of posets $P_1$ and $P_2$, which is the poset $P_1\oplus P_2$ on the disjoint union $P_1 \uplus P_2$ such that $s\preceq t$ in $P_1 \oplus P_2$ if (i) $s,t\in P_1$ and $s\preceq t$ in $P_1$, or (ii) $s,t\in P_2$ and $s\preceq t$ in $P_2$, or (iii) $s\in P_1$ and $t\in P_2$.

\begin{thm}\label{OrdinalSumEhr}
Let $P_1$ and $P_2$ be two disjoint posets. The Ehrhart series of the order polytope of $P_1 \oplus P_2$ is given by
\[
\mathrm{Ehr}(\mathcal{O}(P_1 \oplus P_2),x)=\mathrm{Ehr}(\mathcal{O}(P_1),x)\cdot \mathrm{Ehr}(\mathcal{O}(P_2),x)\cdot (1-x).
\]
\end{thm}
\begin{proof}
Let $\# P_1=n$ and $\# P_2=m$. Then $\# (P_1 \oplus P_2)=n+m$. By \eqref{EhrJordanH}, we just need to prove that
\begin{equation}\label{ProoJHS}
\sum_{\pi \in \mathcal{L}(P_1 \oplus P_2)}x^{d(\pi)}=\sum_{\pi \in \mathcal{L}(P_1)}x^{d(\pi)}\cdot \sum_{\pi \in \mathcal{L}(P_2)}x^{d(\pi)}=\sum_{\pi_1 \in \mathcal{L}(P_1);\ \pi_2 \in \mathcal{L}(P_2)}x^{d(\pi_1)+d(\pi_2)}.
\end{equation}
If $\pi_1=(a_1a_2\cdots a_n)\in \mathcal{L}(P_1)$ and $\pi_2=(b_1b_2\cdots b_m)\in \mathcal{L}(P_2)$, then we define
\[
\varphi (\pi_1;\pi_2)=(a_1a_2\cdots a_n(b_1+n)(b_2+n)\cdots (b_m+n)).
\]
Furthermore, $\varphi (\pi_1;\pi_2)\in \mathcal{L}(P_1 \oplus P_2)$. Conversely, suppose $(u_1u_2\cdots u_nv_1v_2\cdots v_m)\in \mathcal{L}(P_1 \oplus P_2)$. By the definition of $P_1 \oplus P_2$, we have $v_i\geq n+1$ for $1\leq i\leq m$. Let $\pi_1=(u_1u_2\cdots u_n)$ and $\pi_2=((v_1-n)(v_2-n)\cdots(v_m-n))$, we have $\pi_1\in \mathcal{L}(P_1)$, $\pi_2\in\mathcal{L}(P_2)$, and $\varphi(\pi_1;\pi_2)=(u_1u_2\cdots u_nv_1v_2\cdots v_m)$. Therefore, $\varphi$ is a bijection. Moreover, $d(\varphi (\pi_1;\pi_2))=d(\pi_1)+d(\pi_2)$. Hence, Equation \eqref{ProoJHS} holds, which completes the proof.
\end{proof}

An equivalent form of the theorem appeared in \cite[Proposition 6.5]{HaaseKohlTsuchiya} recently. Let $P^{\oplus k}$ denote $P\oplus P\oplus \cdots \oplus P$, with $k$ copies of the $P$'s. Then the following corollary is immediate.
\begin{cor}\label{CorKsumP}
Suppose $p, k\in \mathbb{P}$. Let $P$ be a finite poset of cardinality $p$. If the Ehrhart series of $\mathcal{O}(P)$ is
\[
\mathrm{Ehr}(\mathcal{O}(P),x)=\frac{h^{*}(x)}{(1-x)^{p+1}},
\]
then the Ehrhart series of $\mathcal{O}(P^{\oplus k})$ is
\[
\mathrm{Ehr}(\mathcal{O}(P^{\oplus k}),x)=\frac{(h^*(x))^k}{(1-x)^{pk+1}}.
\]
\end{cor}

\begin{exa}\label{ExampIK}
Let $I_k$ be a $k$-element chain. Then $I_1$ is a $1$-element chain, i.e., an isolated vertex. We know that $\mathrm{Ehr}(\mathcal{O}(I_1),x)=\frac{1}{(1-x)^{2}}$. The poset $I_k$ is the $k$-fold ordinal sum of $I_1$.
By Corollary \ref{CorKsumP}, we have $\mathrm{Ehr}(\mathcal{O}(I_k),x)=\frac{1}{(1-x)^{k+1}}$.
\end{exa}

Now, we introduce a powerful definition in enumerative combinatorics.
\begin{dfn}{\em \cite{LW. Shapiro}}
Let $g(x)=\sum_{n=0}^{\infty}g_nx^n$ and $f(x)=\sum_{n=0}^{\infty}f_nx^n$.
A \emph{Riordan array} is
\[
M=(g(x), f(x))=[g(x), g(x)f(x), g(x)f^{2}(x), \ldots],
\]
whose generating function of the $i$-th column is $M_i(x)=g(x)f^{i}(x)$. The $M=(g(x), f(x))$ is called proper if $g_0=1, f_0=0$ and $f_1\neq 0$. In this case, $M=(g(x), f(x))$ is an infinite lower triangular matrix.
\end{dfn}

For more results on the Riordan array and applications, see \cite{LW. SSBCHMW}. For Corollary \ref{CorKsumP}, we formally define $\mathrm{Ehr}(\mathcal{O}(P^{\oplus 0}),x)=\frac{1}{1-x}$ when $k=0$. We consider $x^k\cdot \mathrm{Ehr}(\mathcal{O}(P^{\oplus k}),x)$ for $k\geq 0$. Now, we can construct a Riordan array $\left(\frac{1}{1-x},\frac{xh^*(x)}{(1-x)^p}\right)$ that is an infinite lower triangular matrix. This will give us many interesting counting results, as shown in the following subsections.

\begin{cor}\label{GluedP1P2}
Let $P_1$ and $P_2$ be two disjoint posets. The poset $P_1$ has a unique maximal element $u$. The poset $P_2$ has a unique minimal element $v$. If the maximal element $u$ of the poset $P_1$ and the minimal element
$v$ of the poset $P_2$ are glued together in the Hasse diagram, denote the new poset as $P_1\diamond P_2$, then the Ehrhart series of $\mathcal{O}(P_1\diamond P_2)$ is given by
\[
\mathrm{Ehr}(\mathcal{O}(P_1\diamond P_2),x)=\mathrm{Ehr}(\mathcal{O}(P_1),x)\cdot \mathrm{Ehr}(\mathcal{O}(P_2),x)\cdot (1-x)^2.
\]
\end{cor}
\begin{proof}
Let $P_1\setminus \{u\}$ be the poset obtained from $P_1$ by deleting the unique maximal element $u$ and its covering relationship follows $P_1$. The Ehrhart series of order polytope of the one-element poset is $\frac{1}{(1-x)^2}$. By Theorem \ref{OrdinalSumEhr}, the Ehrhart series of $\mathcal{O}(P_1)$ is
\[
\mathrm{Ehr}(\mathcal{O}(P_1),x)=\mathrm{Ehr}(\mathcal{O}(P_1\setminus \{u\}),x)\cdot\frac{1}{(1-x)^2}\cdot (1-x).
\]
Since $P_1\diamond P_2=(P_1\setminus \{u\})\oplus P_2$, we have
\[
\mathrm{Ehr}(\mathcal{O}(P_1\diamond P_2),x)=\mathrm{Ehr}(P_1\setminus \{u\}),x)\cdot \mathrm{Ehr}(\mathcal{O}(P_2),x)\cdot (1-x).
\]
This completes the proof.
\end{proof}

\subsection{Some Examples}

Now let us consider some known posets in some literature. These posets have been studied by many scholars for their various properties. We mainly study the Ehrhart series for these posets in this subsection.

\begin{exa}\label{ExampJorH6}
We calculate the Ehrhart series of the order polytope of poset $P$ in Figure \ref{OP1}.
\begin{figure}[H]
\centering
\includegraphics[width=0.23\linewidth]{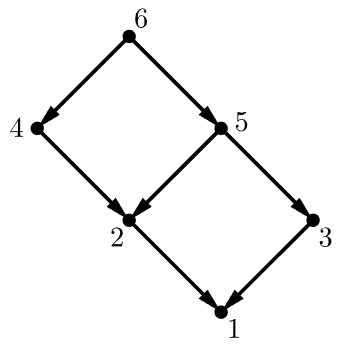}
\caption{The Hasse diagram of poset $P$ in Example \ref{ExampJorH6}.}
\label{OP1}
\end{figure}

The Jordan-H\"older set of this poset is as follows:
\begin{align*}
\pi &= 123456,\ d(\pi)=0; & \pi &=123546,\ d(\pi)=1; & \pi &= 124356,\ d(\pi)=1;
\\ \pi &=132456,\ d(\pi)=1; & \pi &= 132546,\ d(\pi)=2.
\end{align*}
By \eqref{EhrJordanH}, we have
\[ 
\mathrm{Ehr}(\mathcal{O}(P),x)=\frac{1+3x+x^2}{(1-x)^7}.
\] 
This is the generating function of sequence [A006542] on the OEIS \cite{Sloane23}. Hence, we provide a new combinatorial interpretation for this sequence.

Let $P^{\diamond k}$ denote $P\diamond P\diamond \cdots \diamond P$, with $k$ copies of the $P$'s.
Now we consider the $\mathcal{O}(P^{\diamond k})$. By Corollary \ref{GluedP1P2}, we have
\[
\mathrm{Ehr}(\mathcal{O}(P^{\diamond k}),x)=\frac{(1+3x+x^2)^k}{(1-x)^{5k+2}}.
\]
The $\mathrm{Ehr}(\mathcal{O}(P^{\diamond k}),x)$, $k\geq 0$ corresponds to the Riordan array $\left(\frac{1}{(1-x)^2}, \frac{1+3x+x^2}{(1-x)^5}\right)$. The numerators $(1+3x+x^2)^k, k\geq 0$ are the row generating functions of sequence [A272866] on the OEIS.

By the Riordan array $\left(\frac{1}{(1-x)^2}, \frac{1+3x+x^2}{(1-x)^5}\right)$, we can get a combinatorial identity. We have
\[
\mathrm{Ehr}(\mathcal{O}(P^{\diamond i}),x)\cdot \frac{1+3x+x^2}{(1-x)^5}=\mathrm{Ehr}(\mathcal{O}(P^{\diamond (i+1)}),x),
\]
i.e.,
\[
\sum_{n\geq 0}\mathrm{ehr}(\mathcal{O}(P^{\diamond i}),n)x^n\cdot \sum_{n\geq 0}\frac{1}{24}(n+1)(n+2)(5n^2+15n+12)x^n=\sum_{n\geq 0}\mathrm{ehr}(\mathcal{O}(P^{\diamond (i+1)}),n)x^n.
\]
Therefore, we have
\[
\mathrm{ehr}(\mathcal{O}(P^{\diamond (i+1)}),n)=\sum_{j=0}^n\frac{1}{24}(j+1)(j+2)(5j^2+15j+12)\cdot\mathrm{ehr}(\mathcal{O}(P^{\diamond i}),n-j).
\]
\end{exa}

\begin{exa}\label{RadioTower}
D'al\'i and Delucchi \cite{DaliDelucchi} consider the ``radio-tower" poset $X_k$. This is the $k$-fold ordinal sum of the $2$-element antichain. The Ehrhart series of $\mathcal{O}(X_1)$ is
\[
\mathrm{Ehr}(\mathcal{O}(X_1),x)=\sum_{n\geq 0}(n+1)^2x^n=\frac{1+x}{(1-x)^3}.
\]
By Corollary \ref{CorKsumP}, we have
\[
\mathrm{Ehr}(\mathcal{O}(X_k),x)=\frac{(1+x)^k}{(1-x)^{2k+1}}.
\]
The $\mathrm{Ehr}(\mathcal{O}(X_k),x)$, $k\geq 0$ corresponds to the Riordan array $\left(\frac{1}{1-x}, \frac{1+x}{(1-x)^2}\right)$.
\end{exa}

\begin{exa}
Gao, Hou, and Xin \cite{GaoHouXin} studied the $P$-partitions related to ordinal sums of posets. For instance, the poset $H_k$ for hexagonal plane partitions with diagonals, see Figure \ref{OP6}.
\begin{figure}[htp]
\centering
\includegraphics[width=0.65\linewidth]{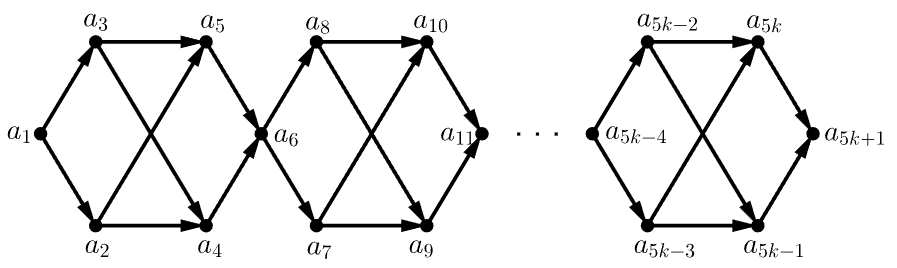}
\caption{The poset $H_k$ for hexagonal plane partitions with diagonals.}
\label{OP6}
\end{figure}

In Example \ref{RadioTower}, the Ehrhart series of $\mathcal{O}(X_2)$ for the $2$-fold ordinal sum of the $2$-element antichain is
\[
\mathrm{Ehr}(\mathcal{O}(X_2),x)=\frac{(1+x)^2}{(1-x)^5}.
\]
By Theorem \ref{OrdinalSumEhr}, we have
\[
\mathrm{Ehr}(\mathcal{O}(I_1\oplus X_2),x)=\frac{(1+x)^2}{(1-x)^6}.
\]
Since $H_k=(I_1\oplus X_2)\oplus(I_1\oplus X_2)\oplus \cdots (I_1\oplus X_2)\oplus I_1$, we obtain
\[
\mathrm{Ehr}(\mathcal{O}(H_k),x)=\frac{(1+x)^{2k}}{(1-x)^{6k}}\cdot (1-x)^{k-1}\cdot \frac{1}{(1-x)^2}\cdot (1-x)=\frac{(1+x)^{2k}}{(1-x)^{5k+2}}.
\]
The $\mathrm{Ehr}(\mathcal{O}(H_k),x)$, $k\geq 0$ corresponds to the Riordan array $\left(\frac{1}{(1-x)^2}, \frac{(1+x)^2}{(1-x)^5}\right)$.
\end{exa}

\begin{exa}
MacMahon \cite{MacMahon1912} considered the Boolean poset of order $3$ described by Figure \ref{OP7}.
\begin{figure}[htp]
\centering
\includegraphics[width=0.3\linewidth]{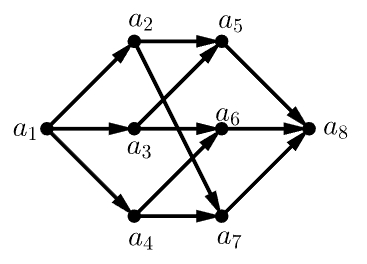}
\caption{The Boolean poset $B_3$.}
\label{OP7}
\end{figure}

\noindent
Through simple calculations, we obtain
\[
\mathrm{Ehr}(\mathcal{O}(B_3),x)=\frac{1+11x+24x^2+11x^3+x^4}{(1-x)^9}.
\]
We glue $B_3$ along their extremal elements to obtain a poset $B_3^{\diamond k}$. By Corollary \ref{GluedP1P2}, we get
\[
\mathrm{Ehr}(\mathcal{O}(B_3^{\diamond k}),x)=\frac{(1+11x+24x^2+11x^3+x^4)^k}{(1-x)^{7k+2}}.
\]
The $\mathrm{Ehr}(\mathcal{O}(B_3^{\diamond k}),x)$, $k\geq 0$ corresponds to the Riordan array $\left(\frac{1}{(1-x)^2}, \frac{1+11x+24x^2+11x^3+x^4}{(1-x)^7}\right)$.
\end{exa}

\begin{exa}
Andrews, Paule, and Riese \cite{Andrews10} studied plane partitions with diagonals whose corresponding poset $PD_k$ (also see \cite{GaoHouXin}) can be depicted in Figure \ref{OP8}. In fact, $PD_k=A\oplus (B \oplus B \oplus \cdots B) \oplus C$, there are $k-1$ $B$'s in the middle. The posets $A,B$ and $C$ given in Figure \ref{OP9}.
\begin{figure}[htp]
\centering
\includegraphics[width=0.9\linewidth]{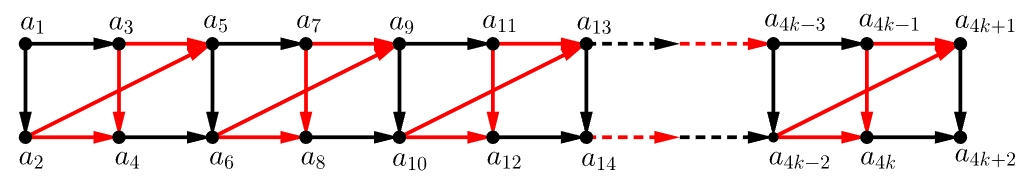}
\caption{The poset $PD_k$ for plane partitions with diagonals.}
\label{OP8}
\end{figure}
\begin{figure}[htp]
\centering
\includegraphics[width=0.8\linewidth]{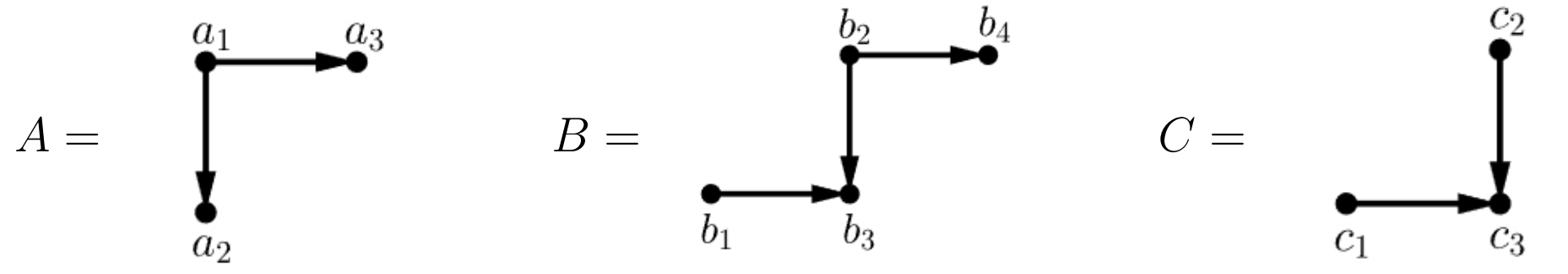}
\caption{Three basic blocks of the poset $PD_k$.}\label{OP9}
\end{figure}

Since the Ehrhart series of $\mathcal{O}(B)$ is
\[
\mathrm{Ehr}(\mathcal{O}(B),x)=\frac{1+3x+x^2}{(1-x)^5}, 
\]
we obtain
\[
\mathrm{Ehr}(\mathcal{O}(PD_k),x)=\frac{(1+x)^2(1+3x+x^2)^{k-1}}{(1-x)^{4k+3}}.
\]
\end{exa}

\section{Direct Product and Ordinal Product}\label{sec-prod}
In this section, we introduce two additional operations that can be performed on posets. Note that these operations are descriptive in nature. That is, the Ehrhart polynomials of the resulting new posets are not determined by the Ehrhart polynomials of the original posets.

\subsection{Direct Product}

If $P_1$ and $P_2$ are two disjoint posets, then the \emph{direct product} of $P_1$ and $P_2$ is the poset $P_1\times P_2$ on $\{(s,t)| s\in P_1, t\in P_2\}$ such that $(s,t)\preceq (s^{\prime},t^{\prime})$ in $P_1\times P_2$ if $s\preceq s^{\prime}$ in $P_1$ and $t\preceq t^{\prime}$ in $P_2$. It is clear from the definition that $P_1\times P_2$ and $P_2\times P_1$ are isomorphic.

\begin{exa}
Let $I_k$ be a $k$-element chain and $V$ be a $3$-element poset in Figure \ref{OP2}. Let $M_k=I_k\times V$.
We consider the $\mathcal{O}(M_k)$. In \cite{Kreweras81}, Kreweras and Niederhausen obtained the number of linear extensions of $M_k$ with exactly $i$ switchbacks (i.e., $i$'s descents), denoted $T(k,i)$, where $0\leq i\leq 2k-1$. For the sequences $T(k,i)$, $0\leq i\leq 2k-1$, see \cite[A140136]{Sloane23}. Therefore, we have
\[
\mathrm{Ehr}(\mathcal{O}(M_k),x)=\frac{\sum_{i=0}^{2k-1}T(k,i)x^i}{(1-x)^{3k+1}}.
\]

From the work of Kreweras and Niederhausen \cite{Kreweras81}, we can find that $\mathrm{Ehr}(\mathcal{O}(M_k),x)$, $k\geq 0$ are the column generating functions of sequence [A111910]. Hence, we have
\[
\mathrm{ehr}(\mathcal{O}(M_k),n)=\frac{(n+k+1)!(2n+2k+1)!}{(n+1)!(2n+1)!(k+1)!(2k+1)!}.
\]
Especially, $\mathrm{Ehr}(\mathcal{O}(M_k),x)$, $k=1,2,3$, correspond to the sequences [A000330], [A006858], and [A006859], respectively. Therefore, we provide some new combinatorial interpretations for the aforementioned sequences.
\end{exa}

\begin{exa}
Let $I_k$ be a $k$-element chain and $\lozenge$ be a $4$-element poset in Figure \ref{OP3}.
\begin{figure}[htp]
\centering
\includegraphics[width=0.2\linewidth]{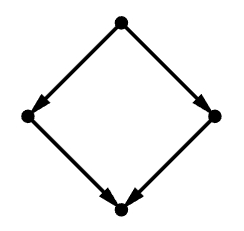}
\caption{The Hasse diagram of $4$-element poset $\lozenge$.}
\label{OP3}
\end{figure}

In fact, $\lozenge=I_2\times I_2$. The $\lozenge$ is called \emph{diamond poset}. The poset $I_k\times \lozenge$ is called multi-cube posets in \cite{DuHou12}. We consider the Ehrhart series of $\mathcal{O}(I_k\times \lozenge)$. We can obtain the following results using the \texttt{CTEuclid} package in \cite{Xin15} (or Stembridge's software package \cite{Stempackage}):
\begin{scriptsize}
\begin{align*}
\mathrm{Ehr}(\mathcal{O}(I_1\times \lozenge),x)&=\frac{1+x}{(1-x)^5},
\\ \mathrm{Ehr}(\mathcal{O}(I_2\times \lozenge),x)&=\frac{1+11x+24x^2+11x^3+x^4}{(1-x)^9},
\\ \mathrm{Ehr}(\mathcal{O}(I_3\times \lozenge),x)&=\frac{1+37x+315x^2+873x^3+873x^4+315x^5+37x^6+x^7}{(1-x)^{13}},
\\ \mathrm{Ehr}(\mathcal{O}(I_4\times \lozenge),x)&=\frac{1+88x+1841x^2+13812x^3+44050x^4+64374x^5+44050x^6+13812x^7+1841x^8+88x^9+x^{10}}{(1-x)^{17}}.
\end{align*}
\end{scriptsize}
The $\mathrm{Ehr}(\mathcal{O}(I_k\times \lozenge),x)$, $k=1,2,3,4$, correspond to the sequences [A002415], [A056932], [A006360], and [A006361], respectively.

By \cite[Proposition 3.5.1]{RP.Stanley} and Equation \eqref{EhrJordanH}, the Ehrhart series $\mathrm{Ehr}(\mathcal{O}(I_k\times \lozenge),x)$ is the generating function of antichains (or order ideals) in the poset $I_2\times I_2\times I_k\times I_n$ or the size of the distributive lattice $J(I_2\times I_2\times I_k\times I_n)$ for each $k$. It is precisely the combinatorial interpretations given on the OEIS.
\end{exa}

\subsection{Ordinal Product}

If $P_1$ and $P_2$ are two disjoint posets, then the \emph{ordinal product} of $P_1$ and $P_2$ is the poset $P_1\otimes P_2$ on $\{(s,t)| s\in P_1, t\in P_2 \}$ such that $(s,t)\preceq (s^{\prime},t^{\prime})$ if (i) $s=s^{\prime}$ and $t\preceq t^{\prime}$, or (ii) $s\prec s^{\prime}$. In general, $P_1\otimes P_2\ncong P_2\otimes P_1$.

It is difficult to obtain the expression for $\mathrm{Ehr}(\mathcal{O}(P_1\otimes P_2),x)$. However, we can obtain some formulas for special cases.

\begin{exa}
Let $I_k$ be a $k$-element chain and $P$ be a finite poset of cardinality $p$. We consider the $\mathcal{O}(I_k\otimes P)$. By the definitions of ordinal sum and ordinal product, it is easy to see that $I_k\otimes P$ is the $k$-fold ordinal sum of $P$. By Theorem \ref{OrdinalSumEhr}, we have
\[
\mathrm{Ehr}(\mathcal{O}(I_k\otimes P),x)=(\mathrm{Ehr}(\mathcal{O}(P),x))^k\cdot (1-x)^{k-1}.
\]
\end{exa}

\begin{prop}
Let $I_k$ be a $k$-element chain and $V$ be a $3$-element poset in Figure \ref{OP2}. Then we have
\[
\mathrm{Ehr}(\mathcal{O}(V\otimes I_k),x)=\frac{\sum_{i=0}^k\binom{k}{i}^2 x^i}{(1-x)^{3k+1}}.
\]
Note that the coefficient of the numerator ($h^{*}$-vector) is the square of the entries of Pascal's triangle, i.e., the Narayana numbers of type B \cite[A008459]{Sloane23}.

\begin{figure}[htp]
\centering
\includegraphics[width=0.2\linewidth]{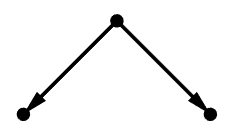}
\caption{The Hasse diagram of $3$-element poset $V$.}
\label{OP2}
\end{figure}
\end{prop}
\begin{proof}
We only need to prove the Ehrhart polynomial
\[
\mathrm{ehr}(\mathcal{O}(V\otimes I_k),n)=\sum_{i=0}^k\binom{k}{i}^2\binom{3k+n-i}{n-i}.
\]
It is obvious that the Ehrhart series of $\mathcal{O}(I_k)$ is
\[
\mathrm{Ehr}(\mathcal{O}(I_k),x) = \frac{1}{(1-x)^{k+1}}.
\]
Observe that $V\otimes I_k=(I_k+I_k)\oplus I_k$.
By Equation \eqref{DirecSumHP} and Theorem \ref{OrdinalSumEhr},
we have
\[
\mathrm{Ehr}(\mathcal{O}(V\otimes I_k),x)=\sum_{i\geq 0}\binom{k+i}{i}^2x^i\cdot \frac{1}{(1-x)^{k+1}}\cdot (1-x).
\]
Therefore, we have
\[
\mathrm{ehr}(\mathcal{O}(V\otimes I_k),n) = \sum_{i=0}^n \binom{k+i}{i}^2 \binom{k+n-i-1}{n-i-1} = \sum_{i=0}^k\binom{k}{i}^2\binom{3k+n-i}{n-i}.
\]
This completes the proof.
\end{proof}

\subsection{An Example}
In this subsection, we give an example to show that $\mathrm{Ehr}(\mathcal{O}(P\times Q),x)$ and $\mathrm{Ehr}(\mathcal{O}(P\otimes Q),x)$ are not determined by $\mathrm{Ehr}(\mathcal{O}(P),x)$ and $\mathrm{Ehr}(\mathcal{O}(Q),x)$. Therefore, to compute these Ehrhart series is usually hard, because we have to go into the structure of new posets.

\begin{figure}[htp]
\centering
\includegraphics[width=0.7\linewidth]{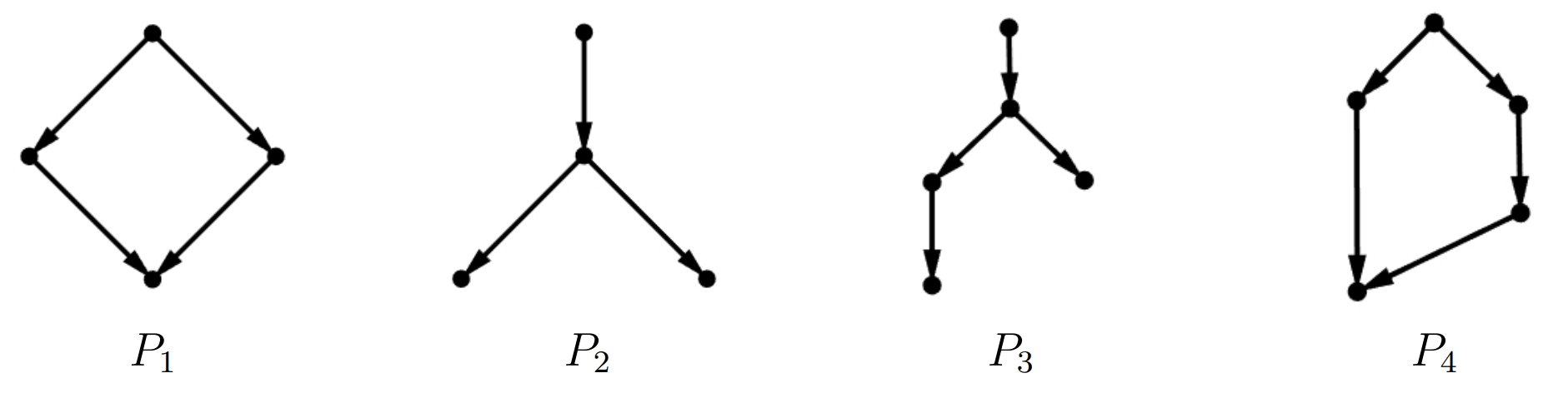}
\caption{The Hasse diagrams of $P_1, P_2, P_3$, and $P_4$.}\label{OP14}
\end{figure}

In Figure \ref{OP14}, we can easily obtain the Ehrhart series for $\mathcal{O}(P_1), \mathcal{O}(P_2), \mathcal{O}(P_3)$, and $\mathcal{O}(P_4)$ as follows:
\begin{align*}
\mathrm{Ehr}(\mathcal{O}(P_1),x)&=\mathrm{Ehr}(\mathcal{O}(P_2),x)=\frac{1+x}{(1-x)^5}
\\ \mathrm{Ehr}(\mathcal{O}(P_3),x)&=\mathrm{Ehr}(\mathcal{O}(P_4),x)=\frac{1+2x}{(1-x)^6}.
\end{align*}
We obtained the following results using Stembridge's software package \cite{Stempackage}:
\begin{tiny}
\begin{align*}
\mathrm{Ehr}(\mathcal{O}(P_1\times P_3),x)&= \frac{6 x^{14}+1511 x^{13}+82568 x^{12}+1588338 x^{11}+13452273 x^{10}+56518447 x^{9}+125835274 x^{8}+153505841 x^{7}}{\left(1-x \right)^{21}}
\\&\quad +\frac{103554312 x^{6}+38139327 x^{5}+7383903 x^{4}+698609 x^{3}+28315 x^{2}+383 x +1}{\left(1-x \right)^{21}},
\\ \mathrm{Ehr}(\mathcal{O}(P_2 \times P_4),x)&=\frac{6 x^{14}+1584 x^{13}+85102 x^{12}+1599870 x^{11}+13290736 x^{10}+55102635 x^{9}+121874741 x^{8}+148652880 x^{7}}{\left(1-x \right)^{21}}
\\&\quad +\frac{100816946 x^{6}+37469554 x^{5}+7329760 x^{4}+699374 x^{3}+28412 x^{2}+381 x +1}{\left(1-x \right)^{21}},
\\ \mathrm{Ehr}(\mathcal{O}(P_3\otimes P_3),x) &= \frac{32x^9+1048x^8+6844x^7+16574x^6+19395x^5+12249x^4+4264x^3+770x^2+59x+1}{(1-x)^{26}},
\\ \mathrm{Ehr}(\mathcal{O}(P_4\otimes P_4),x) &= \frac{4928x^9+47040x^8+152880x^7+232104x^6+186864x^5+83298x^4+20124x^3+2382x^2+108x+1}{(1-x)^{26}}.
\end{align*}
\end{tiny}

It is easy to see that $\mathrm{Ehr}(\mathcal{O}(P_1\times P_3),x)\neq \mathrm{Ehr}(\mathcal{O}(P_2 \times P_4),x)$ and $\mathrm{Ehr}(\mathcal{O}(P_3\otimes P_3),x)\neq \mathrm{Ehr}(\mathcal{O}(P_4\otimes P_4),x)$. This indicates that $\mathrm{Ehr}(\mathcal{O}(P\times Q),x)$ and $\mathrm{Ehr}(\mathcal{O}(P\otimes Q),x)$ may not have concise expressions.

\section{The Poset $I_1\oplus (I_{p_1}+I_{p_2}+\cdots +I_{p_r})\oplus I_1$ and Its Variant}\label{sec-Ipr}

\subsection{The Ehrhart Series}

Let $p_i \in \mathbb{P}$, $1\leq i\leq r$. Let $I_k$ be the $k$-element chain. Consider the order polytope of the poset $I_1\oplus (I_{p_1}+I_{p_2}+\cdots +I_{p_r})\oplus I_1$, which is depicted in Figure \ref{OP4}.
\begin{figure}[htp]
\centering
\includegraphics[width=0.55\linewidth]{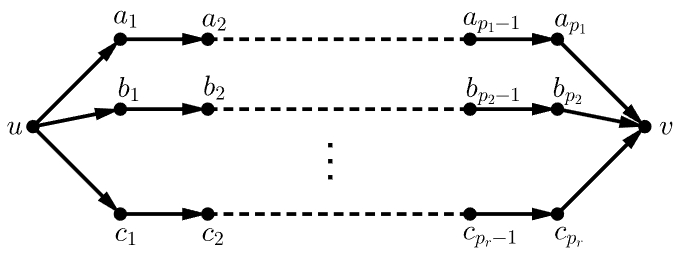}
\caption{The poset $I_1\oplus (I_{p_1}+I_{p_2}+\cdots +I_{p_r})\oplus I_1$.}
\label{OP4}
\end{figure}

Let $M=\{1^{p_1}, 2^{p_2}, \ldots, r^{p_r}\}$ be a multiset, $N=p_1+p_2+\cdots +p_r$, and $\mathfrak{S}_M$ be the set of multipermutations on $M$. Let $T(M, r, i)$ be the number of multipermutations $\pi$ in $\mathfrak{S}_M$ with $\mathrm{des}(\pi)=i$, i.e., $\pi$ has $i$ descents, the positions $j<N$ such that $\pi_j>\pi_{j+1}$.
For example, if $p_1=p_2=2$, $r=2$, then $T(M,2,1)=4$ since all the multipermutations of $1,1,2,2$ with $1$ descent are $1212$, $1221$, $2112$, $2211$.

Define the generating function
\[
A_{M,r}(x)=\sum_{\pi\in \mathfrak{S}_M}x^{\mathrm{des}(\pi)}=\sum_{i\geq 0}T(M, r, i)x^i.
\]

\begin{thm}\label{rkrkline}
Let $p_i \in \mathbb{P}$, $1\leq i\leq r$, and $N=p_1+p_2+\cdots +p_r$. Let $I_k$ be a $k$-element chain. The Ehrhart series of $\mathcal{O}(I_1\oplus (I_{p_1}+I_{p_2}+\cdots +I_{p_r})\oplus I_1)$ is
\[
\mathrm{Ehr}(\mathcal{O}(I_1\oplus (I_{p_1}+I_{p_2}+\cdots +I_{p_r})\oplus I_1),x)
=\frac{A_{M, r}(x)}{(1-x)^{N+3}},
\]
where $A_{M,r}(x)$ is as defined above.
\end{thm}
\begin{proof}
It is simple to see that
\[
\sum_{\pi \in \mathcal{L}(I_1\oplus (I_{p_1}+\cdots +I_{p_r})\oplus I_1)}x^{d(\pi)}=\sum_{\pi\in \mathfrak{S}_M}x^{\mathrm{des}(\pi)}.
\]
The proof is completed by applying \eqref{EhrJordanH}.
\end{proof}

\begin{rem}
By Example \ref{ExampIK}, the Ehrhart series of $\mathcal{O}(I_{p_i})$ is
\[
\mathrm{Ehr}(\mathcal{O}(I_{p_i}),x)=\sum_{n\geq 0}\binom{p_i+n}{n}x^n.
\]
By Theorem \ref{OrdinalSumEhr} and Proposition \ref{PropDirecSum}, the Ehrhart series of $\mathcal{O}(I_1\oplus (I_{p_1}+I_{p_2}+\cdots +I_{p_r})\oplus I_1)$ is
\[
\mathrm{Ehr}(\mathcal{O}(I_1\oplus (I_{p_1}+I_{p_2}+\cdots +I_{p_r})\oplus I_1),x)
=\frac{1}{(1-x)^2}\sum_{n\geq 0}\prod_{l=1}^r\binom{p_l+n}{n}x^n.
\]
In fact, MacMahon \cite{MacMahon19} proved that
\begin{align*}
\frac{A_{M, r}(x)}{(1-x)^{N+1}}=\sum_{n\geq 0}\prod_{l=1}^{r}\binom{p_l+n}{n}x^n,
\end{align*}
where $N$ and $A_{M,r}(x)$ are as defined above. In particular, we have
\[
T(M, r, i)=\sum_{j=0}^i(-1)^j\binom{N+1}{j}\prod_{l=1}^r\binom{p_l+i-j}{p_l}.
\]
\end{rem}

When $p_1=p_2=\cdots =p_r=1$, $T(M, r, i)$ is the Eulerian numbers \cite[Page 38--41]{RP.Stanley}. For $p_1=p_2=\cdots =p_r=1, 2, 3, 4, 5, 6$, $T(M, r, i)$ corresponds to sequences [A008292], [A154283], [A174266], [A236463], [A237202], [A237252] on the OEIS, respectively.

Now, let $p_1=p_2=\cdots =p_r=k\geq 1$. We consider the coefficient sequences of $\mathrm{Ehr}(\mathcal{O}(I_1\oplus (I_{p_1}+I_{p_2}+\cdots +I_{p_r})\oplus I_1),x)$, compared with the sequences on the OEIS, we found some interesting results in the following Table \ref{tabbI1IRI1}. So, we give a new combinatorial interpretation of these sequences about the Ehrhart series.

When $p_1=p_2=\cdots =p_r=1$, suppose $x\cdot \mathrm{Ehr}(\mathcal{O}(I_1\oplus (I_{p_1}+I_{p_2}+\cdots +I_{p_r})\oplus I_1),x)=\sum_{n\geq 0}a(n)x^n$. It is clear that $a(n)=\sum_{i=0}^n(n-i+1)i^r$. Furthermore, $a(n)$ satisfies $a(n)=2a(n-1)-a(n-2)+n^r$, $n\geq 2$. This is a recurrence relation for the second partial sums of $r$-th powers \cite[A101093]{Sloane23}. This is also the meaning of the first row sequence in Table \ref{tabbI1IRI1} on the OEIS.

Similarly, we consider a slight variant. The Ehrhart series of $\mathcal{O}(I_1\oplus (I_{p_1}+I_{p_2}+\cdots +I_{p_r}))$ is given by
\[
\mathrm{Ehr}(\mathcal{O}(I_1\oplus (I_{p_1}+I_{p_2}+\cdots +I_{p_r})),x)
=\frac{A_{M, r}(x)}{(1-x)^{N+2}}.
\]
Let $p_1=p_2=\cdots =p_r=k\geq 1$. This is consistent with some sequences on the OEIS (see Table \ref{tabbkrline}). Further, when $r=1$, it is easy to obtain $\mathrm{ehr}(\mathcal{O}(I_1\oplus I_{k}),n)=\sum_{i=0}^n\binom{k+i}{k}=\binom{n+k+1}{k+1}$. When $r=2$, we get $\mathrm{ehr}(\mathcal{O}(I_1\oplus (I_{k}+I_{k})),n)=\sum_{i=0}^n\binom{k+i}{k}^2$. For the power sums of binomial coefficients, see \cite{Dzhumafil16}.

\begin{tiny}
\begin{table}[htbp]
    	\centering
    	\caption{Some sequences of $\mathrm{Ehr}(\mathcal{O}(I_1\oplus (I_{p_1}+I_{p_2}+\cdots +I_{p_r})\oplus I_1),x)$ with $p_i=k$.}
    	\begin{tabular}{c||c|c|c|c|c|c|c|c|c|c}
    		\hline \hline
$k,r$ & 1 & 2 & 3 & 4 & 5 & 6 & 7 & 8 & 9 &10  \\
    		\hline
$1$ & A000292 & A002415 & A024166 & A101089 & A101092 & A101093 & A250212 & A253636 & A253637 & A253710 \\
    		\hline
$2$ & A000332 & A101094 & --- & --- & --- & --- & --- & --- & --- & --- \\
    		\hline
$3$ & A000389 & --- & --- & --- & --- & --- & --- & --- & --- & --- \\
    		\hline
$4$ & A000579 & --- & --- & --- & --- & --- & --- & --- & --- & --- \\
    		\hline
$5$ & A000580 & --- & --- & --- & --- & --- & --- & --- & --- & --- \\
    		\hline
$6$ & A000581 & --- & --- & --- & --- & --- & --- & --- & --- & --- \\
    		\hline
$7$ & A000582 & --- & --- & --- & --- & --- & --- & --- & --- & --- \\
    		\hline
$8$ & A001287 & --- & --- & --- & --- & --- & --- & --- & --- & --- \\
    		\hline
    	\end{tabular}\label{tabbI1IRI1}
\end{table}
\end{tiny}

\begin{tiny}
\begin{table}[htbp]
    	\centering
    	\caption{Some sequences of $\mathrm{Ehr}(\mathcal{O}(I_1\oplus (I_{p_1}+I_{p_2}+\cdots +I_{p_r})),x)$ with $p_i=k$.}
    	\begin{tabular}{c||c|c|c|c|c|c|c|c|c|c}
    		\hline \hline
$k,r$ & 1 & 2 & 3 & 4 & 5 & 6 & 7 & 8 & 9 &10  \\
    		\hline
$1$ & A000217 & A000330 & A000537 & A000538 & A000539 & A000540 & A000541 & A000542 & A007487 & A023002 \\
    		\hline
$2$ & A000292 & A024166 & A085438 & A085439 & A085440 & A085441 & A085442 & --- & --- & --- \\
    		\hline
$3$ & A000332 & A086020 & A086021 & A086022 & --- & --- & --- & --- & --- & --- \\
    		\hline
$4$ & A000389 & A086023 & A086024 & --- & --- & --- & --- & --- & --- & --- \\
    		\hline
$5$ & A000579 & A086025 & A086026 & --- & --- & --- & --- & --- & --- & --- \\
    		\hline
$6$ & A000580 & A086027 & A086028 & --- & --- & --- & --- & --- & --- & --- \\
    		\hline
$7$ & A000581 & A086029 & A086030 & --- & --- & --- & --- & --- & --- & --- \\
    		\hline
$8$ & A000582 & A234253 & --- & --- & --- & --- & --- & --- & --- & --- \\
    		\hline
$9$ & A001287 & --- & --- & --- & --- & --- & --- & --- & --- & --- \\
    		\hline
    	\end{tabular}\label{tabbkrline}
\end{table}
\end{tiny}

\subsection{Some Special Cases}

\begin{prop}\label{PropEhrCase2}
Let $k\in \mathbb{P}$. Then we have
\[
\mathrm{Ehr}(\mathcal{O}(I_1\oplus (I_{k}+I_{k})),x)=\sum_{n\geq 0}\sum_{i=0}^n\binom{k+i}{k}^2x^n=\frac{\sum_{i=0}^k\binom{k}{i}^2x^i}{(1-x)^{2k+2}}.
\]
\end{prop}
\begin{proof}
The first ``$=$" is directly obtained. For the second ``$=$", we need to prove
\[
\sum_{n\geq 0}\binom{k+n}{k}^2x^n=\sum_{i=0}^k\binom{k}{i}^ix^i\cdot \sum_{n\geq 0}\binom{2k+n}{2k}x^n,
\]
i.e.,
\[
\binom{k+n}{k}^2=\sum_{i=0}^k\binom{k}{i}^2\binom{2k+n-i}{2k}.
\]
The above equation can be verified by Maple.
\end{proof}

By Proposition \ref{PropEhrCase2}, we obtain the following corollary.
\begin{cor}
Let $k\in \mathbb{P}$ and $M=\{ 1^k,2^k\}$. The number of multipermutations $\pi$ in $\mathfrak{S}_M$ with $\mathrm{des}(\pi)=i$ is $\binom{k}{i}^2$, i.e., the square of the entries of Pascal's triangle \cite[A008459]{Sloane23}.
\end{cor}
\begin{proof}
We provide a combinatorial proof. We consider the lattice path from $(0,0)$ to $(k,k)$. We use $1$ stand for the east step and $2$ stand for the north step. Therefore, each $\pi\in \mathfrak{S}_M$ corresponds to a lattice path from $(0,0)$ to $(k,k)$. Each descent (i.e., $21$) corresponds to a peak in the lattice path. We count the number of lattice paths with exactly $i$ peaks. In other words, we need to choose $i$ horizontal and $i$ vertical coordinates. The number is $\binom{k}{i}^2$.
\end{proof}

In fact, when $M=\{ 1^k,2^k,3^k\}$, $k\geq 1$, the number of multipermutations $\pi$ in $\mathfrak{S}_M$ with $\mathrm{des}(\pi)=i$ is the sequence \cite[A181544]{Sloane23}. When $M=\{ 1^k,2^k,3^k,4^k\}$, $k\geq 1$, the number of multipermutations $\pi$ in $\mathfrak{S}_M$ with $\mathrm{des}(\pi)=i$ is the sequence \cite[A262014]{Sloane23}.

Now, let us consider a specific example. If $r=2$ and $k=1$, then we have $\mathrm{Ehr}(\mathcal{O}(I_1\oplus (I_{1}+I_{1})),x)=\frac{1+x}{(1-x)^{4}}$. Let $T_m$ be the $m$-fold ordinal sum of $I_1\oplus (I_{1}+I_{1})$. Then we get $\mathrm{Ehr}(\mathcal{O}(T_m),x)=\frac{(1+x)^m}{(1-x)^{3m+1}}$. We assume $\mathrm{Ehr}(\mathcal{O}(T_0),x)=\frac{1}{1-x}$.

The sequence $[\mathrm{Ehr}(\mathcal{O}(T_0),x), x\cdot\mathrm{Ehr}(\mathcal{O}(T_1),x), x^2\cdot\mathrm{Ehr}(\mathcal{O}(T_1),x),\ldots]$ corresponds to the Riordan array $\left(\frac{1}{1-x}, \frac{x(1+x)}{(1-x)^3}\right)$. We have the following infinite lower triangular matrix:
\[
\overline{M}=\left(\begin{array}{ccccccccc}
1 & \\
1 & 1 \\
1 & 5 & 1 \\
1 & 14 & 9 & 1 \\
1 & 30 & 43 & 13 & 1 \\
1 & 55 & 147 & 88 & 17 & 1 \\
1 & 91 & 406 & 416 & 149 & 21 & 1 \\
1 & 140 & 966 & 1550 & 901 & 226 & 25 & 1 \\
\vdots & \vdots & \vdots & \vdots & \vdots & \vdots & \vdots & \vdots & 1
\end{array}\right).
\]

Furthermore, we have
\[
\frac{1+x}{(1-x)^3}\cdot \mathrm{Ehr}(\mathcal{O}(T_m),x)=\mathrm{Ehr}(\mathcal{O}(T_{m+1}),x),
\]
i.e.,
\[
\sum_{n\geq 0}(n+1)^2x^n\cdot \sum_{n\geq 0}\mathrm{ehr}(\mathcal{O}(T_m),n)x^n=\sum_{n\geq 0}\mathrm{ehr}(\mathcal{O}(T_{m+1}),n)x^n.
\]
Hence, we obtain
\begin{align}\label{FormulehrTm}
\mathrm{ehr}(\mathcal{O}(T_{m+1}),n)=\sum_{i=0}^n\mathrm{ehr}(\mathcal{O}(T_{m}),i)(n-i+1)^2,\ \ \ \ \mathrm{ehr}(\mathcal{O}(T_{0}),n)=1,\ n\geq 0.
\end{align}

The Ehrhart polynomials $\mathrm{ehr}(\mathcal{O}(T_{1}),n)$ and $\mathrm{ehr}(\mathcal{O}(T_{2}),n)$ correspond to [A000330] and [A259181], respectively. The sequence [A259181] is the total number of squares of all sizes in $i\times i$ sub-squares in an $(n+1)\times (n+1)$ grid, $n\geq 0$. The sequence [A000330] simply gives the number of all sizes of squares in an $(n+1)\times (n+1)$ grid, $n\geq 0$. Now, we can further promote it.

\begin{prop}\label{sl2set00}
Let $n\in  \mathbb{N}$ and $m\geq 1$. The Ehrhart polynomial $\mathrm{ehr}(\mathcal{O}(T_{m}),n)$ counts the number of squares of all sizes in $i\times i$ sub-squares of $m-1$ consecutive decomposition in an $(n+1)\times (n+1)$ grid.
\end{prop}
For example, when $m=2, n=2$, we have a grid of $3\times 3$.  We decompose it once. There are $9$ $(1\times 1)$ grids, $4$ $(2\times 2)$ grids, and $1$ $(3\times 3)$ grid. A $2\times 2$ grid also contains $5$ squares, and a $3\times 3$ grid also contains $14$ squares. So we have $\mathrm{ehr}(\mathcal{O}(T_{2}),2)=9+4\times 5+1\times 14=43$.

When $m=3, n=2$, we have a grid of $3\times 3$.  We need to decompose it twice. In the first decomposition, a total of $9$ $(1\times 1)$ grids, $4$ $(2\times 2)$ grids, and $1$ $(3\times 3)$ grid are generated. In the second decomposition, $4$ $(2\times 2)$ grids produce $16$ $(1\times 1)$ grids and $4$ $(2\times 2)$ grids; $1$ $(3\times 3)$ grid produces $9$ $(1\times 1)$ grids, $4$ $(2\times 2)$ grids, and $1$ $(3\times 3)$ grid. So we have $\mathrm{ehr}(\mathcal{O}(T_{3}),2)=9+16+4\times 5+9+4\times 5+14=88$.
\begin{proof}
Firstly, let $T(m,n)$ be the number of squares of all sizes in $i\times i$ sub-squares of $m-1$ consecutive decomposition in an $(n+1)\times (n+1)$ grid. We just need to prove that $T(m,n)$ and $\mathrm{ehr}(\mathcal{O}(T_{m}),n)$ meet the same initial conditions and recursive relationship. For $m=1$, we get
\[
T(1,n)=\mathrm{ehr}(\mathcal{O}(T_{1}),n)=1^2+2^2+3^2+\cdots +(n+1)^2.
\]
For $i>1$, $0\leq j\leq n$, there are $(n-j+1)^2$ squares of $(j+1)\times (j+1)$ in the $(n+1)\times (n+1)$ grid. So we have $T(i+1,n)=\sum_{j=0}^n(n-j+1)^2T(i,j)$. By \eqref{FormulehrTm}, this completes the proof.
\end{proof}

Let us consider another variant. Let $\ell\geq 0$. It is clear that the Ehrhart series of $\mathcal{O}(I_{\ell+1}\oplus (I_k+I_k))$ is
\[
\mathrm{Ehr}(\mathcal{O}(I_{\ell+1}\oplus (I_k+I_k)),x)=\frac{\sum_{i=0}^{k}\binom{k}{i}^2x^i}{(1-x)^{\ell+2k+2}}.
\]
Some sequences on the OEIS corresponding to $\mathrm{Ehr}(\mathcal{O}(I_{\ell+1}\oplus (I_k+I_k)),x)$ are shown in Table \ref{lollipop}.

\begin{tiny}
\begin{table}[htbp]
    	\centering
    	\caption{Some specific sequences of $\mathrm{Ehr}(\mathcal{O}(I_{\ell+1}\oplus (I_k+I_k)),x)$.}
    	\begin{tabular}{c||c|c|c|c|c|c|c|c|c|c|c}
    		\hline \hline
$k,\ell$ & 0 & 1 & 2 & 3 & 4 & 5 & 6 & 7 & 8 &9 &10 \\
    		\hline
$0$ & A000027 & A000217 & A000292 & A000332 & A000389 & A000579 & A000580 & A000581 & A000582 & A001287 & A001288 \\
    		\hline
$1$ & A000330 & A002415 & A005585 & A040977 & A050486 & A053347 & A054333 & A054334 & A057788 & A266561 & --- \\
    		\hline
$2$ & A024166 & A101049 & A101097 & A101102 & A254469 & A254869 & --- & --- & --- & --- & --- \\
    		\hline
$3$ & A086020 & --- & --- & --- & --- & --- & --- & --- & --- & --- & --- \\
    		\hline
    	\end{tabular}\label{lollipop}
\end{table}
\end{tiny}

Especially when $k=1$, we have
\begin{align*}
\mathrm{Ehr}(\mathcal{O}(I_{\ell+1}\oplus (I_1+I_1)),x)&=\frac{1+x}{(1-x)^{\ell+4}}=(1+x)\sum_{n\geq 0}\binom{n+\ell+3}{n}x^n
\\&=1+\sum_{n\geq 1}\left(\binom{n+\ell+3}{n}+\binom{n+\ell+2}{n-1}\right)x^n.
\end{align*}
Inspired by the comments of \cite[A050486]{Sloane23} of Janjic, Pinter, and Lang, we obtain the following results (i.e., Propositions \ref{sl2set1}, \ref{sl2set2} and \ref{sl2set3}). A set $S$ is said to be a \emph{$m$-set} if $|S| = m$.

\begin{prop}\label{sl2set1}
Suppose $\ell\geq 0$, $n\geq 1$. If a $2$-set $Y$ and an $(n+\ell+1)$-set $Z$ are disjoint subsets of an $(n+\ell+4)$-set $X$, then $\mathrm{ehr}(\mathcal{O}(I_{\ell+1}\oplus (I_1+I_1)),n)$ is the number of $(\ell+4)$-subsets of $X$ intersecting both $Y$ and $Z$.
\end{prop}
\begin{proof}
In \cite{Janjic08}, Janji\'{c} gives a general result. Let a set $P=\{P_1, P_2, \ldots, P_r\}$ of positive integers and $m ,t\in \mathbb{N}$, the set $X=\bigcup_{i=1}^{r+1}X_i$, $|X_i|=p_i$ for $1\leq i\leq r$, and an additional $m$-block $X_{r+1}$. So $|X|=p_1+p_2+\cdots +p_r+m$. Janji\'{c} defined the function $F(r, t, P, m)$ to be the number $(r+t)$-subset of $X$ intersecting each $X_i$, $1\leq i\leq r$, and obtained
\[
F(r, t, P, m)=\sum_{I\in [r]}(-1)^{|I|}\binom{|X|-\sum_{i\in I}p_i}{r+t},
\]
where the sum is taken over all subsets of $[r]$.

Let $r=2, p_1=2, p_2=n+\ell+1$, $m=1$. We have $|X|=n+\ell+4$ and $t=\ell+2$. Therefore
\[
F(2,\ell+2,P,1)=\binom{n+\ell+4}{\ell+4}-\binom{n+\ell+2}{\ell+4}=\binom{n+\ell+3}{n}+\binom{n+\ell+2}{n-1}.
\]
This completes the proof.
\end{proof}

\begin{prop}\label{sl2set2}
Suppose $\ell\geq 0$, $n\geq 1$. $2\cdot \mathrm{ehr}(\mathcal{O}(I_{\ell+1}\oplus (I_1+I_1)),n)$ is the number of ways to place $\ell+2$ queens on an $(n+\ell+2)\times (n+\ell+2)$ chessboard so that they diagonally attack each other exactly $\binom{\ell+2}{2}$ times. No two elements of these queens have a common coordinate. The maximal possible attack number $\binom{\ell+2}{2}$ for $\ell+2$ queens is achievable only when all queens are on the same diagonal.
\end{prop}
\begin{proof}
The number of ways to place $\ell+2$ queens on an $(n+\ell+2)\times (n+\ell+2)$ chessboard so that they diagonally attack each other exactly $\binom{\ell+2}{2}$ times is
\[
\left(\binom{n+\ell+2}{\ell+2}+\sum_{i=0}^{n-1}2\cdot \binom{i+\ell+2}{\ell+2}\right)\cdot 2.
\]
It is easy to check that the above formula is equal to $2\cdot \mathrm{ehr}(\mathcal{O}(I_{\ell+1}\oplus (I_1+I_1)),n)$. This completes the proof.
\end{proof}

\begin{prop}\label{sl2set3}
Suppose $\ell\geq 0$, $n\geq 0$.  $\mathrm{ehr}(\mathcal{O}(I_{\ell+1}\oplus (I_1+I_1)),n)$ is the number of $\frac{1}{2^{\ell+2}}$ of $(\ell+4)$-th unsigned column of triangle \cite[A053120]{Sloane23} (T-Chebyshev, rising powers, zeros omitted), triangle of coefficients of the Chebyshev polynomial $T_m(x), m\geq 0$ of the first kind.
\end{prop}
\begin{proof}
The Chebyshev polynomial of the first kind is defined recursively by
\[
T_0(x)=1,\ \ T_1(x)=x,\ \ T_{m+1}(x)=2xT_m(x)-T_{m-1}(x).
\]
Let $T_m(x)=t_0^{(m)}+t_1^{(m)}x+\cdots +t_m^{(m)}x^m$. Then the following formula holds (see \cite{TJ.Rivlin})
\begin{align*}
&t_{m-(2k+1)}^{(m)}=0, k=0,1,\ldots,\left\lfloor\frac{m-1}{2}\right\rfloor,\\
&t_{m-2k}^{(m)}=(-1)^k\sum_{j=k}^{\left\lfloor \frac{m}{2}\right\rfloor}\binom{m}{2j}\binom{j}{k}, k=0,1,\ldots,\left\lfloor\frac{m}{2}\right\rfloor.
\end{align*}
Therefore, the generating function of the $r$-th ($r\geq1$) unsigned column of triangle \cite[A053120]{Sloane23} (T-Chebyshev, rising powers, zeros omitted) is
\[
\sum_{n\geq 0}\sum_{j=n}^{\left\lfloor\frac{r+2n-1}{2} \right\rfloor}\binom{r+2n-1}{2j}\binom{j}{n}x^n.
\]
Now we only need to verify
\[
\sum_{j=n}^{\left\lfloor\frac{\ell+2n+3}{2} \right\rfloor}\binom{\ell+2n+3}{2j}\binom{j}{n}=2^{\ell+2}\left(\binom{n+\ell+3}{n}+\binom{n+\ell+2}{n-1}\right).
\]
We can check that the above formula is correct using Maple. This completes the proof.
\end{proof}

In \cite[A050486]{Sloane23}, one obvious result is that $\mathrm{ehr}(\mathcal{O}(I_{\ell+1}\oplus (I_1+I_1)),n)$ is the $(\ell+2)$-th partial sum of binomials transform of $[1,2,0,0,\ldots]$, i.e.,
\[
\mathrm{ehr}(\mathcal{O}(I_{\ell+1}\oplus (I_1+I_1)),n)=\sum_{i=0}^n\binom{n+\ell+2}{i+\ell+2}b(i),\ b(i)=(1,2,0,0,\ldots).
\]
This is the comments of the sequences $\mathrm{ehr}(\mathcal{O}(I_{\ell+1}\oplus (I_1+I_1)),n)$, $0\leq \ell\leq 9$ (see Table \ref{lollipop}).

\section{The Ferrers Poset $P_{\lambda}$}\label{sec-Ferrers}

Let $m\in \mathbb{N}$, $\lambda \vdash m$ be a partition. In this section, we consider the \emph{Ferrers poset} $P_{\lambda}$ of shape $\lambda$. For example, if $\lambda=(4, 3, 3, 2, 1)$, the \emph{Ferrers diagram} of $\lambda$ is the left figure in Figure\ref{G19}. The right figure depicts the \emph{Ferrers poset} $P_{\lambda}$  of shape $\lambda$ (rotating clockwise by $45$ degrees). This poset has a unique maximal element $a_1$ and four minimal elements $b_1, b_2, b_3, b_4$.

\begin{figure}[htp]
\centering
\includegraphics[width=0.55\linewidth]{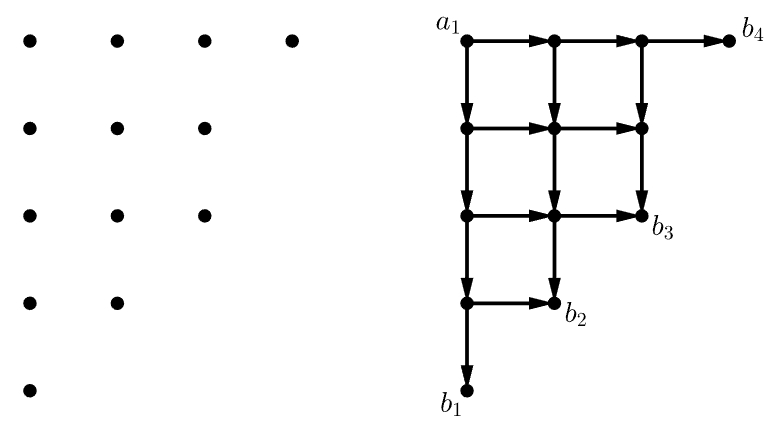}
\caption{The Ferrers diagram and Ferrers poset $P_{\lambda}$ with $\lambda=(4,3,3,2,1)$.}
\label{G19}
\end{figure}

\subsection{The General Case: $\lambda=(\lambda_1,\lambda_2,\ldots,\lambda_t)$.}

We consider the Ehrhart polynomial of the order polytope $\mathcal{O}(P_{\lambda})$. We need the following result from the Lindstr\"om-Gessel-Viennot on non-intersecting lattice paths.

\begin{lem}[Lindstr\"om-Gessel-Viennot, \cite{GesselViennot,Lindstrom}]\label{LGV-noninter-path}
Let $S_1,S_2,\ldots,S_p$ and $E_1,E_2,\ldots,E_p$ be lattice points, with the property that if $1\leq i< j\leq p$ and $1\leq k<\ell \leq p$, then any path from $S_i$ to $E_{\ell}$ must intersect any path from $S_j$ to $E_k$. Then the number of families $(P_1,P_2,\ldots,P_p)$ of non-intersecting lattice paths, where $P_i$ runs from $S_i$ to $E_i$, $i=1,2,\ldots,p$, is given by
\begin{align*}
\det_{1\leq i,j\leq p}(|P(S_j\rightarrow E_i)|)
\end{align*}
where $P(S\rightarrow E)$ denotes the set of all lattice paths from $S$ to $E$.
\end{lem}
The Lindstr\"om-Gessel-Viennot lemma on non-intersecting lattice paths has a wide range of applications. See \cite{Brenti93,Gessel85}.

\begin{thm}\label{GeneralCase}
Let $\lambda=(\lambda_1,\lambda_2,\ldots,\lambda_t)$ be a partition. The Ehrhart polynomial of $\mathcal{O}(P_{\lambda})$ is given by
\begin{align}\label{GeneralFormula}
\mathrm{ehr}(\mathcal{O}(P_{\lambda}),n)=\det_{1\leq i,j\leq t}\left(\binom{n+i+\lambda_{t-j+1}-1}{n+i-j}\right)
=\det_{1\leq i,j\leq t}\left(\binom{n+\lambda_{t-j+1}}{n+i-j}\right).
\end{align}
\end{thm}

For $\lambda=(\lambda_1,\lambda_2,\ldots,\lambda_t)$, the Ehrhart polynomial of the order polytope $\mathcal{O}(P_{\lambda})$ counts the number of plane partitions such that the $i$-th row has at most $\lambda_i$ parts, and the largest part is less than or equal to $n$. The determinant formula for this counting problem is first given by Kreweras \cite[Section 2.3.7]{G.Kreweras}. An equivalent result also appears in MacMahon's book \cite[Page 243]{MacMahon19} and has been frequently rediscovered in various guises. A probabilistic method for this counting problem was given by Gessel \cite{Gessel86}. A linear recursion was given by Pemantle and Wilf \cite{Pemantle09}. When $\lambda=(k-d, k-2d,\ldots ,k-td)$, the determinant can be explicitly evaluated. See \cite[Exercise 7.101(b)]{RP.Stanley99}. A brief introduction to this problem can also be found in \cite[Exercise 149]{RP.Stanley}. For an extensive survey of the evaluation of combinatorial determinants, see \cite{Krattenthaler01,Krattenthaler05},

In this section, we will reestablish the determinant formula by the Lindstr\"om-Gessel-Viennot lemma on non-intersecting lattice paths. We will provide two distinct proofs, each based on different combinatorial interpretations. As applications, we will derive specific results for the two cases $\lambda=(k, k-1,\ldots ,1)$ and $\lambda=(k,k,\ldots,k)$. These two cases allow us to rediscover the determinant formula presented in Theorem \ref{GeneralCase}.

\begin{proof}[The first proof of Theorem \ref{GeneralCase}]
Let us consider the first ``$=$". The left hand side equals $\#\mathcal{L}$, where $\mathcal{L}$ is the set of plane partitions whose $i$-th row has at most $\lambda_i$ parts and whose largest part is less than or equal to $n$. The right hand side equals $\#\mathcal{R}$, where $\mathcal{R}$ is a set to be constructed as follows.
Then we will establish a bijection from $\mathcal{L}$ to $\mathcal{R}$ for the first equality.

Let $S_j=(-\lambda_{t-j+1}-(j-1),j-1), 1\leq j\leq t$ and $E_i=(0,n+i-1), 1\leq i\leq t$ be the starting and end vertices, respectively. Consider $t$ tuples of lattice paths from the $S_j$'s to the $E_i$'s that always move one unit east or north. It is easy to show that the $t$ tuple of lattice paths satisfies the condition of Lemma \ref{LGV-noninter-path}. Apply Lemma \ref{LGV-noninter-path} with the easy fact $|P(S_j\rightarrow E_i)|=\binom{n+i+\lambda_{t-j+1}-1}{n+i-j}$. Then we shall construct $\mathcal{R}$ as the set of $t$ tuples of non-intersecting lattice paths from $S_i$ to $E_i$ for all $i$.

Given a plane partition $P$ in $\mathcal{L}$. It is of the form
\begin{align*}
\begin{matrix}
&a_{1,1}\ \ & a_{1,2}\ \ &\cdots\ \ & a_{1,\lambda_1-1}\ \ & a_{1,\lambda_1}
\\& a_{2,1}\ \ & a_{2,2}\ \ & \cdots\ \ & a_{2,\lambda_2}\ \ &
\\P=\ \ &\cdots\ & &  & &
\\& a_{(t-1),1}\ \ & \cdots & a_{(t-1),\lambda_{t-1}} & &
\\& a_{t,1} & \cdots &  a_{t,\lambda_t} & &.
\end{matrix}
\end{align*}
Observe that each row of $P$ is an ordinary partition (allow zero entries). Denote the $i$-th ordinary partition as $\omega_{i}=(a_{(t-i+1),1},a_{(t-i+1),2},\ldots,a_{(t-i+1),\lambda_{t-i+1}})$ from bottom to top, $1\leq i\leq t$.

We represent each ordinary partition $\omega_i$ as a lattice path from $S_i=(-\lambda_{t-j+1}-(j-1),j-1)$ to $E_i=(0,n+i-1)$ by walking north and east. There are $ab$ squares in the rectangle that delimit possible lattice paths from $(0,0)$ to $(a,b)$. For each lattice path, if we put a dot inside each square underneath the lattice path, then we get the Ferrers diagram of a partition from right to left. In Figure \ref{OP11}, the path corresponds to the partition $(12,10,10,7,2)$. The correspondence is one-to-one.
\begin{figure}[htp]
\centering
\includegraphics[width=0.22\linewidth]{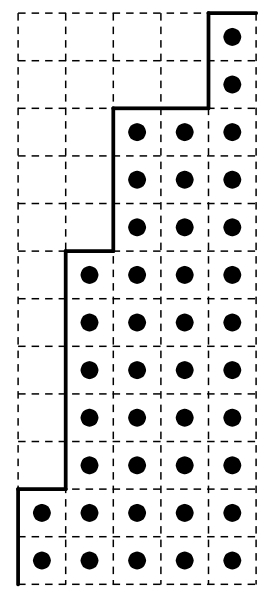}
\caption{The correspondence between lattice path and partition $(12,10,10,7,2)$.}
\label{OP11}
\end{figure}

Because the element in $\mathcal{R}$ is a $t$ tuple of non-intersecting lattice paths, we only need to represent each ordinary partition $\omega_i$ as a lattice path from $S_i=(-\lambda_{t-j+1}-(j-1),j-1)$ to $E_i=(-i+1,n+i-1)$ (If not, two paths must intersect. See Figure \ref{OP12}). As an example, suppose $n=13$, the plane partition
\begin{align*}
&12\ \ 10\ \ 10\ \ \ 7\ \ \ 2
\\&11\ \ \ 9\ \ \ 8\ \ \ \ 6
\\P= \; &11\ \ \ 9\ \ \ 5
\\&6\ \ \ \ 4
\\&3
\end{align*}
corresponds to the tuple of lattice paths given in Figure \ref{OP12}.
\begin{figure}[htp]
\centering
\includegraphics[width=.4\linewidth]{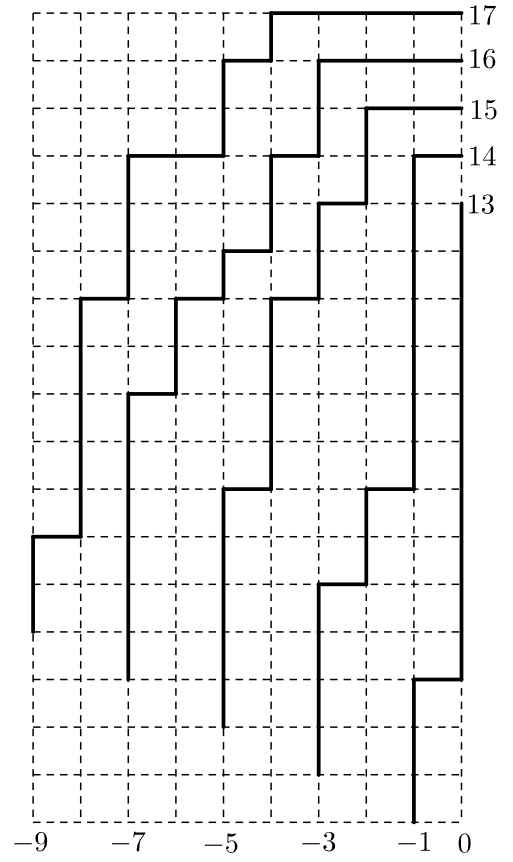}
\caption{The tuple of non-intersecting lattice paths.}
\label{OP12}
\end{figure}

The condition that each column of the plane partition is weakly decreasing is equivalent to the statement that the paths are non-intersecting. Therefore, we established a one-to-one correspondence between the set $\mathcal{L}$ and the set $\mathcal{R}$. This completes the proof of the first ``$=$".

The proof of the second ``$=$" is similar to the first ``$=$". We just need to modify the coordinates of the ending vertices by setting $E_i=(0,n+i-1)$ to $E_i=(-i+1,n+i-1)$. This completes the proof of the second ``$=$".
\end{proof}

Now let us consider the second non-intersecting lattice paths model. This model is inspired by the Jacobi-Trudi identity \cite[Proposition 4.2]{Bressoud}. Before that, we need to introduce the relevant knowledge of symmetric functions \cite[Chapter 7]{RP.Stanley99}.

A \emph{semistandard Young tableaux} (SSYT) of shape $\lambda$ is an array $T=(T_{ij})$ of \textbf{positive integers} of shape $\lambda$ that is weakly increasing in every row and strictly increasing in every column.
If $T$ has $\alpha_i$ parts equal to $i$, then we say that $T$ has {\it type} $\alpha=(\alpha_1, \alpha_2, \ldots)$. Denote $X^{\alpha} := x_1^{\alpha_1}x_2^{\alpha_2}\cdots $.  For instance, a SSYT of shape $\lambda=(4, 3, 3, 2, 1)$ is given by
\begin{align*}
&1\ \ 3\ \ 3\ \ 7
\\&2\ \ 5\ \ 6
\\&4\ \ 6\ \ 7
\\&7\ \ 7
\\&9.
\end{align*}
Furthermore, it has type $\alpha = (1, 1, 2, 1, 1, 2, 4, 0, 1)$, $X^{\alpha}=x_1x_2x_3^2x_4x_5x_6^2x_7^4x_9$.

For a given partition $\lambda=(\lambda_1, \lambda_2, \ldots, \lambda_t) \vdash m$, define the \emph{length} $\ell(\lambda)$ of $\lambda$ by $\ell(\lambda)=t$. The \emph{conjugate partition} $\lambda' = (\lambda'_1, \lambda'_2, \ldots, \lambda'_t)$ of $\lambda$ is defined by the condition that the Ferrers diagram of $\lambda'$ is the transpose of the Ferrers diagram of $\lambda$. For a SSYT  of shape $\lambda$, $T=(T_{ij})$ and a seat $u=(i,j)$, $1\leq j\leq \lambda_i$, the \emph{hook length} $h(u)$ is the number of seats directly to the right or directly below $u$, counting $u$ itself once, i.e., $h(u)=\lambda_i+\lambda^{\prime}_j-i-j+1$. The \emph{content} $c(u)$ is given by $c(u)=j-i$.

\begin{dfn}{\em \cite[Chapter 7]{RP.Stanley99}}
If $\lambda\vdash m$ and $\alpha$ is a weak composition of $m$. Let $K_{\lambda\alpha}$ denote the number of SSYT of shape $\lambda$ and type $\alpha$. $K_{\lambda\alpha}$ is called a Kostka number. The Schur function $s_{\lambda}=s_{\lambda}(X)$ of shape $\lambda$ in the variables $X=(x_1, x_2, \ldots)$ is the formal power series
\[
s_{\lambda}(X)=\sum_{\alpha}K_{\lambda\alpha}X^{\alpha},
\]
summed over all weak compositions $\alpha$ of $m$.
\end{dfn}

Similarly, a \emph{reverse semistandard Young tableaux} (RSSYT) of shape $\lambda$ is an array $T=(T_{ij})$ of positive integers of shape $\lambda$ that is weakly decreasing in every row and strictly decreasing in every column. The number of type $\alpha$ of RSSYT is the same as SSYT.
\begin{prop}{\em \cite[Chapter 7]{RP.Stanley99}}\label{RSSYT}
If $\lambda\vdash m$ and $\alpha$ is a weak composition of $m$. Let $\widehat{K}_{\lambda\alpha}$ be the number of RSSYT of shape $\lambda$ and type $\alpha$. Then $\widehat{K}_{\lambda\alpha}=K_{\lambda\alpha}$. In particular,
\[
s_{\lambda}(X)=\sum_{\alpha}\widehat{K}_{\lambda\alpha}X^{\alpha},
\]
summed over all weak compositions $\alpha$ of $m$.
\end{prop}

Define the \emph{complete homogeneous symmetric functions} $h_{\lambda}$ by the formulas
\begin{align*}
h_k(X) &= \sum_{i_1\leq i_2\leq \cdots \leq i_k}x_{i_1}x_{i_2}\cdots x_{i_k},\ \ k\geq 1\ \ (\text{with } h_0=1) \\
h_{\lambda}(X) &= h_{\lambda_1}(X)h_{\lambda_2}(X)\cdots,\ \ \text{if } \lambda=(\lambda_1,\lambda_2,\ldots).
\end{align*}
Thus $h_k(X)$ is the sum of all monomials of degree $k$.

We consider the specializations of $s_{\lambda}(X)$ and $h_{\lambda}(X)$. The polynomials $s_{\lambda}(x_1, x_2, \ldots, x_n)$ and $h_{\lambda}(x_1, x_2, \ldots, x_n)$ are just the symmetric functions in variables $x_1, x_2, \ldots, x_n$. Moreover, we have
\[
s_{\lambda}(1^n)=s_{\lambda}(x_1, x_2, \ldots, x_n)\big|_{x_1=x_2=\cdots =x_n=1},\ \ \
h_{\lambda}(1^n)=h_{\lambda}(x_1, x_2, \ldots, x_n)\big|_{x_1=x_2=\cdots =x_n=1}.
\]
In fact, by the definition of $h_{\lambda}$,  we have $h_{k}(1^n)=\binom{n+k-1}{k}$.

\begin{lem}{\em \cite[Chapter 7]{RP.Stanley99}}\label{Slam1k}
For any partition $\lambda$ and $n\in \mathbb{P}$, we have
\[
s_{\lambda}(1^n)=\prod_{u\in \lambda}\frac{n+c(u)}{h(u)},
\]
where $c(u)$ and $h(u)$ are defined as above.
\end{lem}

Note that $s_{\lambda}(1^n)=0$ if $n<l(\lambda)$. By Proposition \ref{RSSYT} and Lemma \ref{Slam1k}, we know that $s_{\lambda}(1^n)$ is the number of RSSYT of shape $\lambda$ when the maximum value is less than $n+1$.

\begin{prop}{\em (The Jacobi-Trudi identity) \cite[Proposition 4.2]{Bressoud}}\label{Jacobi-Trudi}
Let $\lambda=(\lambda_1,\lambda_2,\ldots, \lambda_t)$ be a partition. We have that
\[
s_{\lambda}(x_1,x_2,\ldots,x_n)=\det_{1\leq i,j\leq t}(h_{\lambda_i+j-i}).
\]
\end{prop}

\begin{cor}
The number of all reverse semistandard Young tableaux (RSSYT) of shape $\lambda=(\lambda_1,\lambda_2,\ldots, \lambda_t)$ with entries chosen from $\{1,2,\ldots,n\}$ is
\[
s_{\lambda}(1^n)=\prod_{u\in \lambda}\frac{n+c(u)}{h(u)}=\det_{1\leq i,j\leq t}\left(\binom{n+\lambda_i+j-i-1}{n-1} \right).
\]
\end{cor}
\begin{proof}
The corollary follows from Proposition \ref{RSSYT}, Lemma \ref{Slam1k}, and Proposition \ref{Jacobi-Trudi}.
\end{proof}

In \cite[Proposition 4.2]{Bressoud}, the proof of the Jacobi-Trudi identity was completed using the Lindstr\"om-Gessel-Viennot lemma. We observed that RSSYT requires weakly decreasing in every row and strictly decreasing in every column, with entries chosen from $\{1,2,\ldots,n\}$. However, the plane partition requires weakly decreasing in every row, weakly decreasing in every column, and entries chosen from $\{0,1,2,\ldots,n\}$. Inspired by the proof of the Jacobi-Trudi identity \cite[Proposition 4.2]{Bressoud}, we provide the second proof of Theorem \ref{GeneralCase}.

\begin{proof}[The second proof of Theorem \ref{GeneralCase}]
By performing elementary transformations on the determinant to the right of \eqref{GeneralFormula}, we need to prove
\begin{align}\label{GeneralFormuSecond}
\mathrm{ehr}(\mathcal{O}(P_{\lambda}),n)=\det_{1\leq i,j\leq t}\left(\binom{n+\lambda_i}{n+i-j}\right).
\end{align}

Define a \emph{reverse plane partition} (RPP) of shape $\lambda$ is an array $T=(T_{ij})$ of positive integers of shape $\lambda$ that is weakly increasing in every row and weakly increasing in every column.
Let $\mathcal{T}$ be the set of all reverse plane partitions of shape $\lambda=(\lambda_1,\lambda_2,\ldots, \lambda_t)$ with entries chosen from $\{1,2,\ldots,n+1\}$.

For $\lambda=(\lambda_1,\lambda_2,\ldots,\lambda_t)$, the Ehrhart polynomial of the order polytope $\mathcal{O}(P_{\lambda})$ counts the number of plane partitions such that the $i$-th row has at most $\lambda_i$ parts, and the largest part is less than or equal to $n$. The transformation $a_{ij}\mapsto n+1-T_{ij}$ shows that $\mathcal{O}(P_{\lambda})$ is equivalent to counting the number of the elements in set $\mathcal{T}$.

Now the left hand side of \eqref{GeneralFormuSecond} equals $\#\mathcal{T}$. The right hand side \eqref{GeneralFormuSecond} equals $\#\mathcal{D}$, where $\mathcal{D}$ is a set to be constructed as follows.
Then we will establish a bijection from $\mathcal{T}$ to $\mathcal{D}$ to complete the proof of the theorem.

Let $S_j=(j,t-j), 1\leq j\leq t$ and $E_i=(n+i,\lambda_i+t-i), 1\leq i\leq t$ be the starting and end vertices, respectively. Consider $t$ tuples of lattice paths from the $S_j$'s to the $E_i$'s that always move one unit east or north. It is easy to show that the $t$ tuple of lattice paths satisfies the condition of Lemma \ref{LGV-noninter-path}. Apply Lemma \ref{LGV-noninter-path} with the easy fact $|P(S_j\rightarrow E_i)|=\binom{n+\lambda_{i}}{n+i-j}$. Then we shall construct $\mathcal{D}$ as the set of $t$ tuples of non-intersecting lattice paths from $S_i$ to $E_i$ for all $i$.

Given a reverse plane partition $T=(T_{ij})$ in $\mathcal{T}$. It is of the form
\begin{align*}
\begin{matrix}
&T_{1,1}\ \ & T_{1,2}\ \ &\cdots\ \ & T_{1,\lambda_1-1}\ \ & T_{1,\lambda_1}
\\& T_{2,1}\ \ & T_{2,2}\ \ & \cdots\ \ & T_{2,\lambda_2}\ \ &
\\ T=\ \ &\cdots\ & &  & &
\\& T_{(t-1),1}\ \ & \cdots & T_{(t-1),\lambda_{t-1}} & &
\\& T_{t,1} & \cdots &  T_{t,\lambda_t} & &.
\end{matrix}
\end{align*}
Note that $T_{ij}\in \{1,2,\ldots,n+1\}$. Observe that each row of $T$ is an ordinary reverse partition (not allowing zero entries). Denote the $i$-th ordinary reverse partition as $\omega_{i}=(T_{i,1},T_{i,2},\ldots,T_{i,\lambda_i})$ from top to bottom, $1\leq i\leq t$.

There is a natural way of representing each tableaux in $\mathcal{T}$ as a $t$ tuple of lattice paths.
Let $T_{ij}$ be the entry in row $i$, column $j$ of the array $T$. The $i$-th lattice path (counted from the top) encodes $\omega_i$. It goes from $S_i=(i,t-i)$ to $E_i=(n+i,\lambda_i+t-i)$ by walking north and east.
It makes $\lambda_i$ steps to the north, and the $j$-th step to the north is made along the line $x=T_{ij}+i-1$.
As an example, suppose $n=8$ and $\lambda=(6,4,4,2,2,1)$, the reverse plane partition
\begin{align*}
\begin{matrix}
&1\ \ & 1\ \ & 2\ \ & 3\ \ &4\ \ &8\ \
\\ &1\ \ & 2\ \ & 2\ \ & 5\ \ &\ \  &\ \
\\ T=&2\ \ & 2\ \ & 5\ \ & 9\ \ &\ \  &\ \
\\ & 4\ \ & 6 \ \ & \ \ & \ \ &\ \  &\ \
\\ & 5\ \ & 6 \ \ & \ \ & \ \ &\ \  &\ \
\\& 7\ \ &  \ \ &  \ \ &  \ \ &\ \   &\ \ .
\end{matrix}
\end{align*}
corresponds to the tuple of lattice paths given in Figure \ref{OP13}.
\begin{figure}[htp]
\centering
\includegraphics[width=0.55\linewidth]{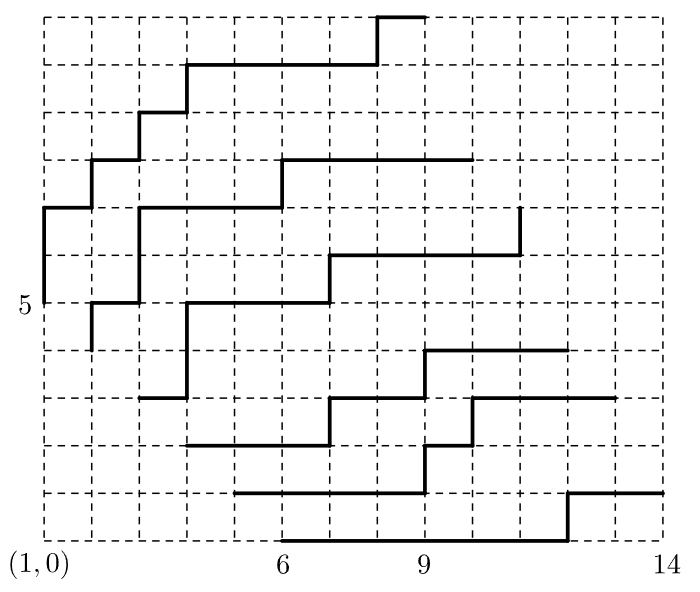}
\caption{The tuple of non-intersecting lattice paths.}
\label{OP13}
\end{figure}

The condition that each column of the reverse plane partition is weakly increasing is equivalent to the statement that the paths are non-intersecting. Therefore, we established a one-to-one correspondence between the set $\mathcal{T}$ and the set $\mathcal{D}$. This completes the proof.
\end{proof}

\subsection{The Case: $\lambda=(k, k-1,\ldots ,1)$.}

\begin{cor}{\em \cite[Exercise 7.101(b)]{RP.Stanley99}}\label{General-Deter}
Let $\lambda=(k-d, k-2d,\ldots ,k-td)$, $k\geq 1$. The Ehrhart polynomial of $\mathcal{O}(P_{\lambda})$ is given by
\begin{small}
\begin{align}
\mathrm{ehr}(\mathcal{O}(P_{\lambda}),n)=\det_{1\leq i,j\leq t}\left(\binom{n+k-id}{n+i-j} \right)=\prod_{u=(i,j)\in \lambda\atop t+c(u)\leq \lambda_i}\frac{n+t+c(u)}{t+c(u)}\cdot \prod_{u=(i,j)\in \lambda \atop t+c(u)>\lambda_i}\frac{(d+1)n+t+c(u)}{t+c(u)},
\end{align}
\end{small}
where $c(u)$ denotes the content of the square $u$.
\end{cor}
\begin{proof}
By \eqref{GeneralFormuSecond}, we obtain the first ``$=$". By \cite[Exercise 7.101(b)]{RP.Stanley99}, we get the second ``$=$".
\end{proof}

\begin{cor}\label{CnjOnekk-1}
Let $\lambda=(k, k-1,\ldots ,1)$, $k\geq 1$. Suppose $\det_{0\leq i,j\leq -1}(C_{i+j+k+1})=1$. The Ehrhart polynomial of $\mathcal{O}(P_{\lambda})$ is given by
\begin{align}\label{CnjOnekFoumu}
\mathrm{ehr}(\mathcal{O}(P_{\lambda}),n)=\det_{0\leq i,j\leq n-1}(C_{i+j+k+1})
\end{align}
where $C_n$ is the well-known Catalan number (\cite[A000108]{Sloane23}).
\end{cor}

The right hand side of Equation \eqref{CnjOnekFoumu} is a special shifted Hankel determinant. To prove the corollary, it is sufficient to show, by the results of \cite[Theorem 4]{Cigler19}, that
\begin{align*}
\mathrm{ehr}(\mathcal{O}(P_{\lambda}),n)=\det_{0\leq i,j\leq k}\left(\binom{n+i+j}{n+i-j}\right)
=\prod_{t=1}^k(n+t)\prod_{s=1}^{k-1}\frac{s!}{(2s+1)!}\prod_{j=0}^{k-2}\prod_{i=3+2j}^{k+j+1}(2n+i).
\end{align*}
The Ehrhart polynomial of the order polytope $\mathcal{O}(P_{(k, k-1,\ldots ,1)})$ counts the number of plane partitions that the $i$-th row has at most $k-i+1$ parts, and the largest part is less than or equal to $n$.
For $1\leq k\leq 5$, $\mathrm{ehr}(\mathcal{O}(P_{\lambda}),n)$ corresponds to sequences [A000027], [A000330], [A006858], [A091962], and [A335857] on the OEIS, respectively.

\begin{proof}[Proof of Corollary \ref{CnjOnekk-1}]
By  the results of \cite[Theorem 4]{Cigler19}, we only need to prove
\begin{align*}
\mathrm{ehr}(\mathcal{O}(P_{\lambda}),n)=\det_{0\leq i,j\leq k}\left(\binom{n+i+j}{n+i-j}\right)=\det_{1\leq i,j\leq k+1}\left(\binom{n+i+j-2}{2j-2}\right).
\end{align*}
By \eqref{GeneralFormuSecond}, we have
\begin{align*}
\mathrm{ehr}(\mathcal{O}(P_{\lambda}),n)=\det_{1\leq i,j\leq k}\left(\binom{n+i+j-1}{2j-1}\right).
\end{align*}
We noticed that the first column of $\det_{1\leq i,j\leq k+1}\left(\binom{n+i+j-2}{2j-2}\right)$ is all $1$. By $\binom{n+m+1}{m+1}-\binom{n+m}{m+1}=\binom{m+n}{m}$ and performing elementary transformation on the determinant, we obtain
\[
\det_{1\leq i,j\leq k+1}\left(\binom{n+i+j-2}{2j-2}\right)=\det_{1\leq i,j\leq k}\left(\binom{n+i+j-1}{2j-1}\right).
\]
This completes the proof.
\end{proof}

\subsection{The Case: $\lambda=(k, k,\ldots ,k)=(k^t)$.}
When $\lambda=(k^t)$, the Ferrers poset $P_{\lambda}$ is the direct product of two chains, i.e., $I_{k} \times I_{\ell(\lambda)}$. It is also referred to as the \emph{box poset} in some literature.

By \eqref{GeneralFormuSecond}, we obtain the following result. This classic result also appears in \cite[Theorem 3.6]{Bressoud}.
\begin{cor}{\em \cite[Theorem 3.6]{Bressoud}}
Let $k\geq 1$. Let $\lambda=(k^t)$.
The Ehrhart polynomial of $\mathcal{O}(P_{\lambda})$ is given by
\begin{align*}
\mathrm{ehr}(\mathcal{O}(P_{\lambda}),n)=\det_{1\leq i,j\leq t}\left(\binom{n+k}{n+i-j} \right).
\end{align*}
\end{cor}

In this case, there is a simple connection between plane partitions and SSYT, so Stanley's hook-content formula
\cite[Chapter 7]{RP.Stanley99} applies to give the following result.
\begin{thm}\label{hhsnthmschu}
Let $k\geq 1$. Let $\lambda=(k,k,\ldots,k)$.
The Ehrhart polynomial of $\mathcal{O}(P_{\lambda})$ is given by
\[
\mathrm{ehr}(\mathcal{O}(P_{\lambda}),n)=s_{\lambda}(1^{n+\ell(\lambda)})=\prod_{u \in \lambda}\frac{n+\ell(\lambda)+c(u)}{h(u)},
\]
where $\ell(\lambda)$ is the length of $\lambda$, $h(u)$ and $c(u)$ are the hook length and the content of $u \in \lambda$, respectively.
\end{thm}
\begin{proof}
For a given poset $P_{\lambda}$, by Definition \ref{Definitorderp}, the Ehrhart polynomial $\mathrm{ehr}(\mathcal{O}(P_{\lambda}),n)$ counts the number of labels that satisfy the conditions for a poset.
For a given label, we add $\ell(\lambda)$ to the labels in the first row, $\ell(\lambda)-1$ to the labels in the second row, and so on. Thus the value of the last row should be added by $1$. At this point, the labeled poset is an RSSYT of shape $\lambda$. The $s_{\lambda}(1^{n+\ell(\lambda)})$ is the number of RSSYT of shape $\lambda$ when the maximum value is less than $n+\ell(\lambda)+1$. We get $\mathrm{ehr}(\mathcal{O}(P_{\lambda}),n)=s_{\lambda}(1^{n+\ell(\lambda)})$. By Lemma \ref{Slam1k}, this completes the proof.
\end{proof}

\begin{cor}[MacMahon formula, \cite{MathStackExc,Guttmann98}]\label{MacMahonPKT}
Let $k\geq 1$. Let $\lambda=(k^t)$. The Ehrhart polynomial of $\mathcal{O}(P_{\lambda})$ is given by
\[
\mathrm{ehr}(\mathcal{O}(P_{\lambda}),n)=\prod_{i=1}^t\prod_{j=1}^k\frac{i+j+n-1}{i+j-1}.
\]
\end{cor}
\begin{proof}
For any $u = (i, j) \in \lambda=(k^t)$, we have $h(u) = k-i+t-j+1$ and $c(u) = j-i$. By Theorem \ref{hhsnthmschu}, the Ehrhart polynomial of $\mathcal{O}(P_{\lambda})$ is
\[
\mathrm{ehr}(\mathcal{O}(P_{\lambda}),n)= \prod_{i=1}^t\prod_{j=1}^k\frac{n+t+j-i}{k-i+t-j-1} = \prod_{i=1}^t\prod_{j=1}^k\frac{n+k-j+i}{i+j-1}=\prod_{i=1}^t\prod_{j=1}^k\frac{i+j+n-1}{i+j-1}.
\]
This completes the proof.
\end{proof}

\begin{tiny}
\begin{table}[htbp]
    	\centering
    	\caption{Some specific sequences of $\mathrm{Ehr}(\mathcal{O}(P_{\lambda}),x)$, $\lambda=(k,k,\ldots,k)$.}
    	\begin{tabular}{c||c|c|c|c|c|c|c|c|c|c}
    		\hline \hline
$\ell(\lambda),k$  & 2 & 3 & 4 & 5 & 6 & 7 & 8 & 9 & 10 &11 \\
    		\hline
$2$  & A002415 & A006542 & A006857 & A108679 & A134288 & A134289 & A134290 & A134291 & A140925 & A140934 \\
    		\hline
$3$  & --- & A047819 & A107915 & A140901 & A140903 & A140907 & A140912 & A140918 & A140926 & A140935 \\
    		\hline
$4$  & --- & --- & A047835 & A140902 & A140904 & A140908 & A140913 & A140919 & A140927 & A140936 \\
    		\hline
$5$  & --- & --- & --- & A047831 & A140905 & A140909 & A140914 & A140920 & A140928 & A140937 \\
    		\hline
$6$  & --- & --- & --- & --- & A140906 & A140910 & A140915 & A140921 & A140929 & A140938 \\
    		\hline
$7$  & --- & --- & --- & --- & --- & A140911 & A140916 & A140922 & A140930 & A140939 \\
    		\hline
$8$  & --- & --- & --- & --- & --- & --- & A140917 & A140923 & A140931 & A140940 \\
    		\hline
$9$  & --- & --- & --- & --- & --- & --- & --- & A140924 & A140932 & A140941 \\
    		\hline
$10$ & --- & --- & --- & --- & --- & --- & --- & --- & A140933 & A140942 \\
    		\hline
$11$ & --- & --- & --- & --- & --- & --- & --- & --- & --- & A140943 \\
    		\hline
    	\end{tabular}\label{LambdaHHK}
\end{table}
\end{tiny}

In Table \ref{LambdaHHK}, we list some sequences on the OEIS for $\mathrm{Ehr}(\mathcal{O}(P_{\lambda}),x)$, $\lambda=(k,k,\ldots,k)$. It is clear that Table \ref{LambdaHHK} is symmetric.

Observe all comments for a total of 43 sequences from  [A140901] to [A140943] on the OEIS, Berselli would like to count the number of $t\times k$ matrices with elements in $\{0,1,\ldots, n\}$, where each row and each column is in nondecreasing order. It is clear that this number is exactly $\mathrm{ehr}(P_{\lambda},n)$ with $\lambda=(k^t)$. But the specific formulas for these sequences given in \cite{Sloane23} are either empirical or conjectural. Now we provide the proof.

\begin{prop}{\em (Conjectured in \cite[A140934]{Sloane23})}\label{Proposit1409}
Let $k\geq 1$. Let $\lambda=(k^t)$. The Ehrhart polynomial of $\mathcal{O}(P_{\lambda})$ is given by
\[
\mathrm{ehr}(\mathcal{O}(P_{\lambda}),n)=\prod_{i=0}^{k-1}\binom{n+t+k-1}{n+i}\frac{i!}{(n+t+k-i-1)^{k-i-1}}.
\]
\end{prop}
\begin{proof}
We have
\begin{align*}
&\prod_{i=0}^{k-1}\binom{n+t+k-1}{n+i}\frac{i!}{(n+t+k-i-1)^{k-i-1}}
\\=&\prod_{i=1}^k\frac{(n+k+t-1)!(i-1)!}{(n+i-1)!(t+k-i)!(n+t+k-i)^{k-i}}
\\=&\prod_{i=1}^k\frac{(n+k+t-1)!(i-1)!}{(n+i-1)!(t+i-1)!(n+t+i-1)^{i-1}}
\\=&\prod_{i=1}^t\frac{(n+k+t-1)!(i-1)!}{(n+i-1)!(k+i-1)!(n+k+i-1)^{i-1}}.
\end{align*}
By Corollary \ref{MacMahonPKT}, we only need to prove
\[
\prod_{i=1}^t\frac{(n+k+t-1)!(i-1)!}{(n+i-1)!(k+i-1)!(n+k+i-1)^{i-1}}-\prod_{i=1}^t \prod_{j=1}^k\frac{i+j+n-1}{i+j-1}=0.
\]
This equation is proved by Maple. This completes the proof.
\end{proof}

A \emph{standard Young tableaux} (SYT) of shape $\lambda$ is a SSYT of shape $\lambda$ with the type $\alpha=(1^{|\lambda|})$. In \cite{RP.Stanley99}, a descent of a SYT $T$ is an integer $i$ such that $i+1$ appears in a lower row of than $i$. Let $D(T)$ be the set of all descents of $T$.

\begin{lem}{\em \cite[Chapter 7]{RP.Stanley99}}\label{Slamdescdt}
For any partition $\lambda\vdash m \geq 1$, $n\in \mathbb{P}$, we have
\[
s_{\lambda}(1^n)=\sum_{T}\binom{n-d(T)+m-1}{m},
\]
where $d(T)=\#D(T)$, summed over all SYT of shape $\lambda$.
\end{lem}

\begin{cor}
Let $k\geq 1$. Let $\lambda=(k^t)$. The Ehrhart polynomial of $\mathcal{O}(P_{\lambda})$ is given by
\[
\mathrm{ehr}(\mathcal{O}(P_{\lambda}),n)=\sum_{T}\binom{n+(k+1)\cdot t-d(T)-1}{k\cdot t},
\]
where summed over all SYT of shape $\lambda$.
\end{cor}

When $\lambda=(k,k)$, the poset $P_{\lambda}$ is called the \emph{ladder poset}. We obtain the following results.
\begin{thm}\label{EhrhOkk-Naray}
Let $\lambda=(k,k)$ with $k\geq 1$. Then
\begin{align*}
\mathrm{Ehr}(\mathcal{O}(P_{\lambda}),x)=\frac{\sum_{i=1}^kT(k,i)x^{i-1}}{(1-x)^{2k+1}},
\end{align*}
where $T(k,i)=\frac{1}{i}\binom{k-1}{i-1}\binom{k}{i-1}, 1\leq i\leq k$, is the triangle of the Narayana numbers \cite[A001263]{Sloane23}.
\end{thm}
\begin{proof}
By Corollary \ref{MacMahonPKT}, we get
\begin{align*}
\mathrm{ehr}(\mathcal{O}(P_{\lambda}),n)&=\frac{(n+k+1)!(n+k)!}{n!k!(n+1)!(k+1)!} = \frac{1}{n+k+1}\binom{n+k+1}{k+1}\binom{n+k+1}{k}.
\end{align*}
Since
\[
T(k, i)=\frac{1}{i}\binom{k-1}{i-1}\binom{k}{i-1} \quad \text{ and } \quad \frac{1}{(1-x)^{2k+1}}=\sum_{n\geq 0}\binom{2k+n}{n}x^n,
\]
the theorem holds from
\[
\sum_{i=1}^k\frac{1}{i}\binom{k-1}{i-1}\binom{k}{i-1}\binom{2k+n-i+1}{n-i+1}
=\frac{1}{n+k+1}\binom{n+k+1}{k+1}\binom{n+k+1}{k},
\]
which can be verified using Maple.
\end{proof}

The Narayana numbers have rich literature \cite{NYLi08,Slanke98,Slanke00}.

\begin{rem}
Let $\lambda=(k,k)$ with $k\geq 1$. The Ehrhart polynomial of $\mathcal{O}(P_{\lambda})$ is the dimension of the space of the semi-invariants of weight $n$ in \cite[Proposition 8.4]{Mukai}. This gives the alternative formula
\[
\mathrm{ehr}(\mathcal{O}(P_{\lambda}),n)=\binom{k+n+1}{n}^2-\binom{k+n+2}{n+1}\binom{k+n}{n-1}.
\]

The formula $\mathrm{ehr}(\mathcal{O}(P_{\lambda}),n)=\frac{1}{n+k+1}\binom{n+k+1}{k+1}\binom{n+k+1}{k}$ is also the number of permutations of $n+k+1$ that avoid the pattern $132$ and have exactly $k$ descents (see \cite{Hyatt-Remmel}).
\end{rem}

\begin{cor}{\em (Conjectured by Berselli in \cite[A140934]{Sloane23})}
Let $\lambda=(11,11)$.

\begin{enumerate}

\item The $\mathrm{ehr}(\mathcal{O}(P_{\lambda}),n)$ is the $12$-th column of the triangle \cite[A001263]{Sloane23}.

\item The Ehrhart series of $\mathcal{O}(P_{\lambda})$ is given by
\begin{footnotesize}
\begin{align*}
\mathrm{Ehr}(\mathcal{O}(P_{\lambda}),x)=\frac{1+55x+825x^2+4950x^3+13860x^4+19404x^5+13860x^6+4950x^7+825x^8+55x^9+x^{10}}{(1-x)^{23}}.
\end{align*}
\end{footnotesize}

\item The Ehrhart polynomial of $\mathcal{O}(P_{\lambda})$ is given by
\[
\mathrm{ehr}(\mathcal{O}(P_{\lambda}),n)=\frac{n+12}{12n+12}\binom{n+11}{11}^2.
\]

\item The Ehrhart polynomial of $\mathcal{O}(P_{\lambda})$ satisfies
\[
\mathrm{ehr}(\mathcal{O}(P_{\lambda}),n)=\prod_{i=1}^{11}\frac{[A002378](n+i)}{[A002378](i)}.
\]
\end{enumerate}
\end{cor}

\begin{proof}
The sequence [A001263] is the triangle of Narayana numbers. By Theorem \ref{EhrhOkk-Naray}, we obtain
\begin{small}
\[
\mathrm{ehr}(\mathcal{O}(P_{(11,11)}),n)=\frac{1}{n+12}\binom{n+12}{12}\binom{n+12}{11}
=\frac{1}{12}\binom{n+11}{11}\binom{n+12}{11}
=\frac{n+12}{12n+12}\binom{n+11}{11}^2.
\]
\end{small}
We have completed the proof of parts (1) and (3). The Ehrhart series of $\mathcal{O}(P_{\lambda})$ is obtained by Theorem \ref{EhrhOkk-Naray}. The $[A002378](i)=i(i+1)$ is the oblong number. By Corollary \ref{MacMahonPKT}, this completes the proof of part (4).
\end{proof}

\begin{cor}{\em (Conjectured by Eldar in \cite[A140934]{Sloane23})}
Let $\lambda=(11,11)$. The Ehrhart polynomial of $\mathcal{O}(P_{\lambda})$ satisfies
\begin{align*}
\sum_{n\geq 0}\frac{1}{\mathrm{ehr}(\mathcal{O}(P_{\lambda}),n)}&=\frac{3538258540001}{8820}-40646320\cdot \pi^2,
\\ \sum_{n\geq 0}\frac{(-1)^n}{\mathrm{ehr}(\mathcal{O}(P_{\lambda}),n)}&=\frac{1678950598}{2205}-23068672\cdot \frac{\log 2}{21}.
\end{align*}
\end{cor}
\begin{proof}
By $\mathrm{ehr}(\mathcal{O}(P_{(11,11)}),n)=\frac{n+12}{12n+12}\binom{n+11}{11}^2$, we can directly complete the calculation using Maple.
\end{proof}

Let us consider another special case.
\begin{prop}
Suppose $k\geq 1$. Let $\lambda=(k^k)$. The $\mathrm{ehr}(\mathcal{O}(P_{\lambda}),n)$ is the number of tilings of an $\langle k,n,k\rangle$ hexagon \cite[A103905]{Sloane23}.
\end{prop}
\begin{proof}
By \cite[A103905]{Sloane23}, we only need to prove
\begin{align*}
\mathrm{ehr}(\mathcal{O}(P_{\lambda}),n)=\prod_{i=1}^n\frac{\prod_{t=k}^{2k-1}(i+t)}{\prod_{j=0}^{k-1}(i+j)}.
\end{align*}
By Theorem \ref{hhsnthmschu}, this completes the proof.
\end{proof}

\section{The Poset $\bar{P}_{(k,k)}$ of Incomplete $P_{(k,k)}$}\label{sec-Pkk}

We consider the poset $\bar{P}_{(k,k)}$ of \emph{incomplete} $P_{(k,k)}$, defined by Figure \ref{G06}.
\begin{figure}[htp]
\centering
\includegraphics[width=0.4\linewidth]{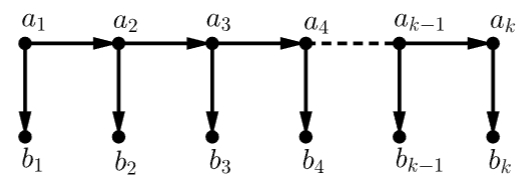}
\caption{The poset $\bar{P}_{(k,k)}$.}
\label{G06}
\end{figure}

\begin{thm}\label{stirlssnk}
Let $k\geq 1$. The Ehrhart polynomial of $\mathcal{O}(\bar{P}_{(k,k)})$ is given by
\begin{align*}
\mathrm{ehr}(\mathcal{O}(\bar{P}_{(k,k)}),n)=S(n+k+1,n+1),
\end{align*}
where $S(r,t)$ is a Stirling number of the second kind \cite[Chapter 1]{RP.Stanley}.
\end{thm}
\begin{proof}
By observing Figure\ref{G06}, if $a_1=i$, $b_1$ has $i+1$ options from $0$ to $i$, we can easily obtain the following recursion:
\[
\mathrm{ehr}(\mathcal{O}(\bar{P}_{(k,k)}),n)=\sum_{i=0}^n(i+1) \mathrm{ehr}(\mathcal{O}(\bar{P}_{(k-1,k-1)}),i), \qquad \mathrm{ehr}(\mathcal{O}(\bar{P}_{(1,1)}),n)=\binom{n+2}{2}.
\]

In \cite[Chapter 1]{RP.Stanley}, the $S(n,k)$ satisfies the recursion $S(n,k)=kS(n-1,k)+S(n-1,k-1)$ for $n\geq 1$, and $S(n+2,n+1)=\binom{n+2}{2}$. Therefore, we have
\[
S(n+k+1,n+1)=(n+1)S(n+k,n+1)+S(n+k,n)=\sum_{i=0}^n(i+1)S(i+k,i+1).
\]
Furthermore, $\mathrm{ehr}(\mathcal{O}(\bar{P}_{(k,k)}),n)$ and $S(n+k+1,n+1)$ satisfy the same recursion and the same initial conditions. This completes the proof.
\end{proof}

\begin{thm}\label{stirlssnkkEuler}
Let $k\geq 1$. The Ehrhart series of $\mathcal{O}(\bar{P}_{(k,k)})$ is given by
\begin{align*}
\mathrm{Ehr}(\mathcal{O}(\bar{P}_{(k,k)}),x)=\frac{\sum_{i=1}^kT(k,i)x^{i-1}}{(1-x)^{2k+1}},
\end{align*}
where $T(k,i), 1\leq i\leq k$ is the second-order Eulerian triangle \cite{Gessel78}. The $T(k,i)$ is the sequence [A008517] on the OEIS.
\end{thm}
\begin{proof}
By \cite{Gessel78}, we have 
\[
\sum_{n=0}^{\infty} S(n+k,n) x^n=\frac{\sum_{i=1}^kT(k,i)x^{i}}{(1-x)^{2k+1}},
\] 
where $S(n+k, n)$ is a Stirling number of the second kind. Then the theorem holds by Theorem \ref{stirlssnk}.
\end{proof}

For $1\leq k\leq 6$, $\mathrm{ehr}(\mathcal{O}(\bar{P}_{(k,k)}),n)$ corresponds to sequences [A000217], [A001296], [A001297], [A001298], [A112494], and [A144969] on the OEIS, respectively.

{\small \textbf{Acknowledgements:}
We are grateful to Monica M.Y. Wang of Beijing Institute of Graphic Communication and Jun Ma of Shanghai Jiao Tong University for many useful suggestions. This work was partially supported by NSFC(12071311).

\end{document}